\documentclass[a4,11pt]{article}

\usepackage{amsthm}
\usepackage{amsmath}
\usepackage{bm}
\usepackage{amssymb}
\usepackage{amsfonts}
\usepackage{comment}
\usepackage{cases}
\usepackage[normalem]{ulem}
\usepackage{algorithm}
\usepackage{algpseudocode}
\usepackage{graphicx}
\usepackage{xcolor}

\usepackage{hyperref}
\usepackage{cleveref}

\numberwithin{equation}{section}

\usepackage{xcolor}

\newcommand{\Div}{\mbox{div}\mspace{3mu}}

\newcommand{\ud}{\,\mbox{d}}

\newcommand{\R}{\mathbb{R}}

\newcommand{\CR}{\mbox{\tiny{CR}}}

\newtheorem{theorem}{Theorem}[section]

\newtheorem{remark}[theorem]{Remark}
\newtheorem{lemma}{Theorem}[section]

\crefname{theorem}{Theorem}{Theorems}
\crefname{assertion}{Assertion}{Assertions}
\crefname{remark}{Remark}{Remarks}

\newcommand{\email}[1]{\texttt{#1}}

\newcommand{\TheTitle}{%
A Two-Stage Finite Element Approach for High-precision Guaranteed Lower Eigenvalue Bounds
}

\title{{\TheTitle}\thanks{%
This work was supported in part by Grant Numbers JP20KK0306, JP24K00538, JP21H00998 and the Japan–China bilateral joint research project (JPJSBP120237407).%
}}

\author{%
Xuefeng Liu\thanks{School of Arts and Sciences, Tokyo Woman's Christian University,
2 Chome-6-1 Zenpukuji, Suginami-ku, Tokyo 167-8585, Japan
(\email{xfliu@lab.twcu.ac.jp}).}
\and
Michael Plum\thanks{Karlsruhe Institute of Technology, Institute for Analysis, Englerstra{\ss}e 2, 76131 Karlsruhe, Germany
(\email{michael.plum@kit.edu}).}
}

\begin{document}

\date{}
\maketitle
\begin{abstract}

Obtaining high-precision guaranteed lower eigenvalue bounds remains difficult, even though the standard high-order conforming finite element (FEM) easily yields extremely sharp upper bounds. Recently developed rigorous  approaches using such as Crouzeix–Raviart or linear conforming elements do not extend well to high-order FEM. Some non-standard FEM approaches can provide sharp eigenvalue bounds but are technically involved. This persistent gap between accurate upper bounds and equally sharp rigorous lower bounds via standard high-order conforming FEMs makes the problem technically demanding and highly competitive. In this paper, we propose a new two-stage rigorous algorithm that closes this gap by employing high-order FEM on graded meshes and producing {\em rigorous lower eigenvalue bounds as sharp as the corresponding high-order upper bounds}, as demonstrated in our numerical examples. Numerical experiments for the Laplacian and Steklov eigenvalue problems on square and dumbbell domains show the accuracy and efficiency of the method, particularly on graded or highly nonuniform meshes. These results confirm that the proposed approach provides a practical and competitive solution to the long-standing difficulty of obtaining sharp, reliable lower eigenvalue bounds.
\end{abstract}

\section{Introduction}

\medskip

Obtaining high-precision guaranteed lower eigenvalue bounds remains difficult, even though standard high-order conforming finite element methods (FEM) easily yield extremely sharp upper bounds.
To take advantage of high-order FEMs, Carstensen et al. have developed several non-standard FEM approaches to obtain high-precision lower bounds \cite{carstensen2020skeletal,carstensen2021guaranteed,carstensen2023direct}. These methods are effective but technically involved and not easy to apply.

In this paper, we instead apply the classical Lehmann--Goerisch theorem in combination with standard high-order FEM to obtain lower eigenvalue bounds of quality comparable to the well-known Rayleigh--Ritz upper bounds. Generally, the Lehmann--Goerisch theorem by itself is not straightforward to use because it requires substantial spectral pre-information, namely a rough lower bound for some higher eigenvalue. 
To address this difficulty, a homotopy method was proposed by the second author (see \cite{plum1990eigenvalue,plum1991bounds} for the initial ideas, and further developments in \cite[Chapter 10]{nakao2019numerical_book}), and applied to many examples, but this approach may involve substantial additional processes and numerical effort. 
In contrast, we introduce a much simpler approach: the projection-based lower eigenvalue bounds obtained via, e.g., the Crouzeix--Raviart FEM are used to provide the necessary spectral pre-information for the Lehmann–Goerisch theorem, thereby avoiding the homotopy method; consequently, the use of standard high-order FEM within the Lehmann--Goerisch framework produces high-precision, fully guaranteed lower eigenvalue bounds.

\medskip

Eigenvalue problems are among the most fundamental subjects in numerical analysis, with a long history of methods for upper and lower bound estimation; see \cite{Babuska-Osborn-1991} for foundational theory and \cite{Boffi2010} for a survey of FEM-based eigenvalue techniques.
Below, we review the literature on rigorously validated eigenvalue bounds that do not rely on additional assumptions and place our method in the broader context.

The homotopy method developed by the second author \cite{plum1991bounds} sets up a family of 
eigenvalue problems that connects the target eigenvalue problem to a base problem with known  spectral bounds, 
and then rigorously estimates the variation of the eigenvalues in this homotopy process. 
This rather general method is applicable to self-adjoint eigenvalue problems also with essential spectrum, and hence, particularly to partial differential equation eigenvalue problems also on unbounded domains. 
To apply the homotopy method, one needs to construct a base problem, which is often available through theoretical techniques.  
The homotopy method takes advantage of the Lehmann--Goerisch theorem  along with rough eigenvalue bounds obtained from the respective previous homotopy step to produce sharp eigenvalue bounds (see early papers \cite{Lehmann1963,behnke1994inclusions,GoerischHe1990,Albrecht1987,plum1990eigenvalue,plum1991bounds,plum1997guaranteed} and a systematic introduction in \cite[Chapter 10]{nakao2019numerical_book}).

The projection-based method is another useful approach to provide explicit eigenvalue bounds when the projection error can be explicitly estimated \cite{Birkhoff_etal1966, Kikuchi+Liu2007,liu-kikuchi-2010, Liu-Oishi-2013,Liu-2015,Kobayashi2015,carstensenGallistl2014, Carstensen2014}. In \cite{Birkhoff_etal1966}, Birkhoff et al. obtained eigenvalue bounds for smooth Sturm--Liouville systems by using projection to piecewise-cubic polynomials. In  \cite{LiuOis2011} and \cite{Liu-Oishi-2013},  the idea of Birkhoff was extended to propose guaranteed two-sided bounds for the Laplacian eigenvalue problem by using the Galerkin projection associated to conforming finite element methods. The proposed eigenvalue bounds can naturally handle problems defined over bounded polygonal domains of arbitrary shape. The projection-error based method is further considered for general self-adjoint differential operators with compact resolvent \cite{Liu-2015} and  eigenvalue problems formulated by positive semidefinite bilinear forms \cite{You-Xie-Liu-2019,liu2020explicit}, where essential spectrum may appear. Meanwhile, the nonconforming finite element method is utilized to provide easy-to-obtain estimation of the projection error. Under efforts from different research groups, the projection-based method proposed has been successfully applied to various eigenvalue problems, including the Laplacian \cite{Liu-2015,Carstensen2014},  the Biharmonic operator \cite{carstensenGallistl2014,liu2018you,liao2019optimal}, the Stokes operator \cite{xie2018explicit,liu2021explicit}, the Steklov operator \cite{You-Xie-Liu-2019,NakanoLiYueLiu+2023} and the Maxwell operator \cite{gallistl2021computational}. 
It is worth pointing out that, since the projection-based method cannot take advantage of high-order finite element methods or graded meshes, it is costly to apply this kind of method with mesh refinement for the purpose of high-precision eigenvalue bounds. Such a defect urges us to turn to the Lehmann--Goerisch theorem for sharper eigenvalue bounds.
Note that the above methods, together with the Rayleigh--Ritz method providing upper eigenvalue bounds, give two-sided eigenvalue bounds with the information about the 
indices of the enclosed eigenvalues.
Hence, eigenvalue-excluding ranges are automatically obtained.

 Another interesting approach is to verify existence and non-existence ranges of eigenvalues in the case that no eigenvalue index information is available,  for example,  for non-self-adjoint problems. 
  Nakao et al. \cite{nakao1999numerical}, Lahmann et al. \cite{lahmann2004computer} and Watanabe et al. \cite{watanabe2009computer, watanabe2011computer,watanabe2014verified} developed methods that provide eigenvalue bounds by identifying ranges where eigenvalues exist and cannot exist. The eigenvalue problem is transformed into an investigation of solution existence and uniqueness for certain nonlinear  differential equations, for which  techniques as described in  \cite{NAKAO1992,plum1993numerical,nakao2019numerical_book} can be used.
  In Nagatou et al. \cite{nagatou2012eigenvalue},  a Birman--Schwinger approach is used for identifying eigenvalue exclusion ranges. 
 See also the general description in \cite[Section 9.3]{nakao2019numerical_book}.

In the field of FEMs, determining two-sided eigenvalue bounds for differential operators is also a significant and difficult task.
For special domains with well-constructed meshes, it has been proved that approximate eigenvalues obtained by the mass-lumping method provide lower eigenvalue bounds directly  \cite{Hu-Huang-Shen-2004}. Many nonconforming finite element methods (FEMs) also yield asymptotic lower bounds for eigenvalues when the eigenfunctions exhibit singularities. That is, as the mesh size tends to zero, the computed eigenvalues converge to the exact values from below. See the early work \cite{armentano2004asymptotic} and the surveys in \cite{Yang2010, Luo-Lin-Xie-2012}. 
However, it is still not clear how to verify  the required precondition, i.e., the mesh size being small enough, to ensure the asymptotic lower bounds.
Recently, Carstensen et al. developed non-standard FEM approaches to provide direct eigenvalue bounds {that can take  advantage of non-uniform (graded) meshes }\cite{carstensen2020skeletal,carstensen2021guaranteed, carstensen2023direct}. The utilization of such approaches leads to a significant increase in degrees of freedom and complexity of the finite element spaces as compared to standard finite element methods.

\medskip

\paragraph{Contribution of this paper} 	
This paper proposes a new two-stage algorithm to provide high-precision guaranteed eigenvalue bounds for differential operators. Different from the projection-based  eigenvalue bounds, the utilization of Lehmann--Goerisch's theorem in the proposed algorithm takes advantage of high-order FEMs and graded meshes.

In the first stage of the proposed algorithm, in order to deal with the eigenvalue problem formulated by positive semidefinite bilinear forms, e.g., the Steklov eigenvalue problem, a new concise proof is provided for the lower eigenvalue bounds proposed in \cite{You-Xie-Liu-2019} (see also \cite{Kikuchi+Liu2007,Kobayashi2015,Carstensen2014,Liu-Oishi-2013}):
\begin{equation}
\label{main-result-summary}
\lambda_k \ge  \frac{\lambda_{k,h}}{1+C_h^2 \lambda_{k,h}}\:.
\end{equation}
Here, $\lambda_{k,h}$ is the exact $k$th eigenvalue of {the} discretized matrix eigenvalue problem obtained from the finite element method (see \S\ref{section-steklov} for details), and 
$C_h$ is a quantity that can be evaluated {explicitly}.
{Here, the nonconforming FEM is helpful, since the quantity $C_h$ can then easily be determined by the error constant for the local interpolation operator associated to the FEM space.}
The new proof simplifies the one of \cite{You-Xie-Liu-2019}, which considers a complicated space decomposition with respect to the kernel space of the operator. Another feature of the proof is that the assumption of positive definiteness of $a(\cdot, \cdot)$ on $V(h)$ (see definition in \S\ref{subsec:lower-bound}) is not needed any more; see Remark \ref{remark:direct-lower-bound}.

\medskip

In the second stage, based on the rough eigenvalue bound \cref{main-result-summary} obtained from the first stage with  low computational intensity, high-precision eigenvalue bounds are provided by utilizing the Lehmann--Goerisch theorem along with high-order conforming FEMs. The Lehmann--Goerisch theorem is a general theoretical framework that provides sharp and optimal eigenvalue bounds, given rough {\it a priori} eigenvalue bounds and high-quality approximate eigenfunctions. Sharp eigenvalue bounds derived from the Lehmann–Goerisch theorem in combination with FEM have been studied in previous works \cite{liu-2013-2,liu2014high,tomas2018}. The application of   the Lehmann--Goerisch theorem to the Laplacian and the Steklov eigenvalue problems are discussed in Sections 4 and 5.

\medskip

We summarize the advantages of the proposed two-stage algorithm for the purpose of high-precision eigenvalue bounds.
\begin{itemize}
    \item [(a)] The utilization of high-order FEM spaces leads to high-precision eigenvalue bounds even for coarse meshes. 
    For example, for the unit square domain, by using the second-order Lagrange FEM over an $8\times 8$ mesh, we {\em prove} by our two-stage algorithm that the first Steklov eigenvalue is bounded as follows:
$$
\underline{0.2400790}83 \le \lambda_1 \le \underline{0.2400790}91  \:.
$$
In our numerical examples in \S \ref{subsec:numerical-results-laplacian}, we consider a problem with nearby eigenvalues, where the high-precision bounds based on the Lehmann--Goerisch theorem are accurate enough to separate the enclosing intervals and thus prove strict inequality between these eigenvalues, while the rough projection based eigenvalue bounds computed with the same domain triangulation are far too inaccurate for such a separation.
	\item [(b)] The Lehmann--Goerisch theorem can take advantage of graded meshes, which is not possible in the projection-based lower eigenvalue bounds in \eqref{main-result-summary}. 
	\item [(c)] According to numerical evidence, the utilization of Lehmann--Goerisch's theorem  provides lower eigenvalue bounds of the optimal convergence rate for the eigenvalues of interest, compared with the possible sub-optimal bounds from the projection-error based method. For example, the explicit lower eigenvalue bound \cref{main-result-summary} when applied to the Steklov eigenvalue problem has a convergence rate  $O(h)$ since $C_h=O(\sqrt{h})$ ($h$: the mesh size) even for convex domains, which is suboptimal compared with the convergence rate of $\lambda_{k,h}$ itself; see details in Remark \ref{remark:sub-optimal-rate-steklov}. 
	 On the contrary, the eigenvalue bound using the linear Lagrange FEM in the second stage of the proposed algorithm has the convergence rate  $O(h^2)$. 

\end{itemize}

\medskip

The rest of the paper is organized as follows. In \S 2, we introduce the abstractly formulated eigenvalue problems and their basic properties. In \S 3, as the first stage of the proposed algorithm, the projection-based  eigenvalue bounds are introduced along with a new and concise proof. There are details on the application of this eigenvalue bounds to the Laplacian and the Steklov eigenvalue problems. In \S 4, as the second stage, high-precision eigenvalue bounds based on the Lehmann--Goerisch theorem are discussed. In \S 5, computational results are presented to demonstrate the efficiency of our proposed algorithm for bounding eigenvalues. Finally, in \S 6, we summarize the results of this paper.

\section{Variational eigenvalue problem}
First, let us formulate the assumptions for the target eigenvalue problem to be discussed.

Let $V$ and $W$ be separable Hilbert spaces endowed with the norms $\|\cdot\|_a$ and $\|\cdot\|_b$ induced by the  inner products $a(\cdot, \cdot)$ and $b(\cdot, \cdot)$, respectively.
Let $\gamma: V\to W$ be a compact linear operator and consider the following eigenvalue problem.

\medskip

Find $u\in V\setminus\{0\}$ and $\lambda \in \mathbb{R}$ such that
\begin{equation}
	\label{eq:objective-evp}
	a(u,v)=\lambda b(\gamma u, \gamma v) \quad \forall v \in V.
\end{equation}

To study this eigenvalue problem, let us introduce the solution operator  $G:W\to V$ such that for any $f\in W$, $ G f \in V$ uniquely solves the following equation:
$$
a(G f, v) = b(f,\gamma v)\quad \forall v \in V.
$$
It is easy to see that {$ G $ is bounded and hence} $\mathcal{K} := G  \gamma : V\to V$ is a compact self-adjoint {positive semidefinite operator}. The eigenvalue problem for the operator $\mathcal{K} $ is to find eigenfunctions $u_k \in V\setminus\{0\}$ and corresponding eigenvalues $\mu_k \ge 0 $ such that 
\begin{equation}
  \label{eq-main-eig-form}
  b(\gamma u_k, \gamma v) = \mu_k \, a(u_k, v)  \quad \forall v \in V.
\end{equation}

\begin{remark}\label{re:evp-examples}
We now illustrate the above framework with two concrete eigenvalue problems.
For the Laplacian eigenvalue problem on a bounded domain 
$\Omega \subset \mathbb{R}^k$, we take the following setting.
\begin{equation}
\label{eq:setting-laplacian-EVP}
  \left\{
    \begin{array}{l} 
      \displaystyle V=H^1_0(\Omega),  \quad W=L^2(\Omega)~; \\
      \displaystyle a(u,v)=\int_{\Omega} \nabla u \cdot \nabla v  \,\mbox{d}x, \quad b(u,v)=\int_{\Omega} u v  \,\mbox{d}x;\\
      \gamma : \mbox{embedding operator}.  
    \end{array}
  \right.
\end{equation}
For the Steklov eigenvalue problem on a bounded Lipschitz domain $\Omega \subset \mathbb{R}^k$, we use the following setting.
\begin{equation}
\label{eq:setting-steklov-EVP}
  \left\{
    \begin{array}{l} 
      \displaystyle V=H^1(\Omega),  \quad W=L^2(\partial \Omega)~; \\
      \displaystyle a(u,v)=\int_{\Omega} \nabla u \cdot \nabla v + uv \,\mbox{d}x, \quad b(u,v)=\int_{\partial \Omega} u v  \,\mbox{d}s;\\
      \gamma : \mbox{trace operator}.  
    \end{array}
  \right.
\end{equation}
\end{remark}

\medskip

\paragraph{\bf Spectrum of $\mathcal{K} $}
Let us summarize standard results based on the spectral theory of compact symmetric operators.
In the case where $\operatorname{Ker}(\mathcal{K} )=\{0\}$,
the eigenvalues of $\mathcal{K} $ are given by a finite (when $\mbox{dim}(V)<\infty$) or countable (when $\mbox{dim}(V)=\infty$) sequence of nonzero $\mu_k$'s such that 
$$
\mu_k \ge \mu_{k+1}(>0) \mbox{ for }k\ge 1; \quad  \lim_{k\to \infty}  \mu_k =0 \mbox{ if } \mbox{dim}(V)=\infty ~.
$$
Note that these eigenvalues are repeated according to their multiplicities. 
In the case where  $\mbox{dim}(\operatorname{Ker}(\mathcal{K} ))>0$,  
$b(\gamma ~ \cdot ,\gamma ~ \cdot )$ is only positive semidefinite on $V$, and in addition to the nonzero $\mu_k$'s, $0$  is also an eigenvalue of $\mathcal{K} $, possibly of infinite multiplicity.

\paragraph{\bf Space decomposition of $V$}
Note that 
$\mbox{Ker}( \mathcal{K}  ) (\subset V)$ can be characterized by \footnote{Note that for $u$ such that $b(\gamma u, \gamma u)=0$, we have $\gamma u=0$, since $b(\cdot,\cdot)$ is the inner product on  $W$. Hence,  $\mathcal{K} u =G \gamma u=0$. On the other side, for $u$ such that $\mathcal{K} u=0$, we have
$b(\gamma u, \gamma  v)= a(\mathcal{K} u, v)=0$ $ \forall v\in V$, and hence 
$b(\gamma u, \gamma  u)=0$. }
$$
\mbox{Ker}( \mathcal{K}  ) = \{u \in V \:|\: b(\gamma u, \gamma u) =0 \} = \{u \in V \:|\: \gamma u =0 \}\:.
$$
Then $V$ admits an $a$-orthogonal decomposition:
$V = \mbox{Ker}(\mathcal{K} ) \oplus  \mbox{Ker} (\mathcal{K} )^\perp
$. 

\medskip

We  introduce the Rayleigh quotient $R(\cdot)$ on  $V$: for $v\in V$, $v\neq 0$,
\begin{equation}
\label{eq:rayleigh_quotient}
R(v):=\frac{b(\gamma v,\gamma v)}{a(v,v)}.
\end{equation}
The eigenvalues $\mu_k$ are the stationary values of $R$ on $V$.
The following  principles play an important role in studying eigenvalue problems (see, e.g., \cite{kato1966, weinberger1974variational,nakao1999numerical}).
  \begin{equation}
   \label{eq:min-max-eig-subspace}
    \mu_{k}=\max_{\substack{H_{k}\subset V}}~\min_{\substack{v\in H_{k}\setminus \{0\}}}R(v)=
	\min_{\substack{S_{k-1}\subset V}} ~\max_{\substack{v\perp S_{k-1},v\neq 0}} R(v)~.
  \end{equation}
Here, $H_k$ denotes a $k$-dimensional subspace of $V$, and
 $S_{k}$ denotes a subspace of $V$  with $\mbox{dim}(S_{k}) \le k$. 

Let $\{u_1,u_2, \ldots\}$ be an $a$-orthonormal system of eigenfunctions associated to the $\mu_{k}$'s. 
Define $E_{k}:=\mbox{span}\{u_1, \ldots, u_k\}$.
Then the eigenvalue $\mu_k$ can be characterized as follows:
\begin{equation}
	\mu_{k}=\min_{\substack{v\in E_{k}\setminus \{0\}}}R(v) = \max_{\substack{v\in V\setminus \{0\}, ~v\perp E_{k-1}}} R(v), 
  \label{eq:min-max-eigen-subspace}
\end{equation}

\section{{First stage:} projection-based  explicit eigenvalue bounds}

As the first stage of our proposed algorithm, we present the projection-based method that produces easy-to-implement lower bounds for the target eigenvalues. Such a method has been well studied in a series of papers by the first author \cite{Liu-Oishi-2013,Liu-2015,liu2020explicit}. See also  \cite{Birkhoff_etal1966,Kikuchi+Liu2007,Kobayashi2015,Carstensen2014}.

 Let us recall the target eigenvalue problem:  Find $u \in V$ and $\mu \ge 0$ such that
\begin{equation}
\label{eq:new-formulation-eig}
b(\gamma u, \gamma v) = \mu a(u,v)\quad \forall v \in V\:.
\end{equation}
Eigenvalue problems for linear partial differential operators are usually formulated for the reciprocals of the nonzero eigenvalues $\mu_k$. Denote $\lambda_k = 1/\mu_k$. Then 
\begin{equation}
	\label{eq:eig-problem-with-lambda}
  0 < \lambda_1 \leq \lambda_2 \leq \cdots,
  \quad\text{and}\quad 
  a(u_k, v) = \lambda_k b(\gamma u_k, \gamma v)\quad \forall v \in V~.
\end{equation}

\subsection{Discretized eigenvalue problems and lower eigenvalue bounds}
\label{subsec:lower-bound}
Let $V_h$ be a finite-dimensional space for the purpose of  approximating $V$; $\mbox{dim}(V_h)=d (<\infty)$. It is assumed that $V_h$ and $V$ inherit the same addition and scalar multiplication from a certain larger vector space so that $V(h):=V+V_h$ is properly defined. 
Let $\widehat{a}(\cdot, \cdot)$ be a positive semidefinite bilinear form on $V(h)$ and whose restriction to $V_h$ is an inner product on  $V_h$. It is supposed that $\widehat{a}(\cdot, \cdot)$ is an extension of $a(\cdot, \cdot)$ to $V(h)$ in the sense that 
$$
\widehat{a}(u, v ) = a(u, v )~~\mbox{ for any } u, v \in V\:.
$$

Note that 
$\widehat{a}(\cdot, \cdot)$ is not necessarily positive definite on $V(h)$. An example of $\widehat{a}(\cdot, \cdot)$ being not positive definite on $V(h)$ is as follows \footnote{Thanks to Ryoki Endo for the construction of this example.}: $V=\{(x,0)\:|\:x\in \R\}$, $V_h=\{(0,y)\:|\:y\in \R\}$,  $V(h)=V+V_h=\mathbb{R}^2$,  
$\widehat{a}(u,v)=u^{\mathsf{T}} \begin{pmatrix}
 1 & 2 \\ 2 & 4 \end{pmatrix} v$ for $u,v\in V(h)$. 
The seminorm or norm induced by $\widehat{a}$ is still denoted by 
$\|\cdot\|_a$.

Assume that the operator $\gamma$ is extended to a compact linear operator $\widehat{\gamma}:V(h)\to W$, i.e., $\widehat{\gamma} u = \gamma u$  for $u \in V$.
In the rest of the paper, 
we henceforth write $a(\cdot, \cdot)$ and $\gamma$ for $\widehat a(\cdot, \cdot)$ and $\widehat\gamma$ when no confusion can arise.
 
 \medskip

 \paragraph{Discretized eigenvalue problem}
The approximation of the eigenvalue problem (\ref{eq:new-formulation-eig}) on $V_h$ is given by: Find $\mu_{h} \ge 0$ and $u_{h} \in V_{h}\setminus \{0\}$, such that
\begin{eqnarray}
\label{eq:eig-problem-n-m-h}
b(\gamma u_{h}, \gamma v_{h}) = \mu_h a(u_{h}, v_{h})\quad\forall v_{h}\in V_{h}~.
\end{eqnarray}
The approximate eigenvalues are denoted as
$$
\mu_{1,h} \ge \mu_{2,h} \ge \cdots \ge \mu_{{d'},h} > \mu_{{d'}+1,h} = \cdots  = 
\mu_{d,h}=0~.
$$
Here, $(d-{d'})$ is the dimension of the kernel of ${b(\gamma \cdot,\gamma \cdot)}$ on $V_h$. 
Let $\{u_{k,h}\}_{k=1}^{d}$ be $a$-orthonormal eigenfunctions
associated to the eigenvalues $\mu_{k,h}$.
Define $E_{k,h}:=\mbox{span}\{u_{1,h}, \cdots, u_{k,h}\}$.
 
Corresponding to \cref{eq:eig-problem-with-lambda}, we also consider its reciprocal  version: 
Find $\lambda_{h} > 0 $ and $u_{h} \in V_{h}, \gamma u_h \neq 0$ such that
\begin{eqnarray}
\label{eq:eig-problem-m-n-h}
a(u_{h},v_{h}) = \lambda_h b(\gamma u_{h}, \gamma v_{h})\quad\forall v_{h}\in V_{h}.
\end{eqnarray}
The above problem has ${d'}$ positive eigenvalues 
\begin{equation}
	\label{eq:def-lambda_k_h}
\lambda_{1,h} \le \lambda_{2,h}\le \cdots \le \lambda_{{d'},h}; \quad \lambda_{k,h}=1/\mu_{k,h}\quad (1\le k \le {d'})\:.
\end{equation}

\medskip

\medskip

The max-min and min-max principles for eigenvalues tell that
\begin{eqnarray}
&&    \mu_{k,h}=\max_{\substack{H_{k,h}\subset V_h}} ~\min_{\substack{v_h\in H_{k,h}\setminus \{0\}}}{R}(v_h) 
=\min_{\substack{S_{k-1,h}\subset V_h}} ~\max_{\substack{v_h\in V_h\setminus \{0\}, \\ ~v_h\perp S_{k-1,h}}} {R}(v_h)\:.
\label{eq:min-max-eig-subspace-fem}
\end{eqnarray}
 Here, $H_{k,h}$ denotes a $k$-dimensional subspace of  $V_h$,  and 
  $S_{k-1,h}$ denotes a subspace of  $V_h$ such that $\mbox{dim}(S_{k-1,h}) \le k-1$.  With the eigen-subspace  $E_{k,h}$, we also have
\begin{eqnarray}
&&\mu_{k,h}=\min_{\substack{v_h\in E_{k,h}\setminus \{0\}}} R(v_h) = 
\max_{\substack{v_h\in V_h\setminus\{0\}\\ v_h\perp E_{k-1,h}}} R(v_h) \:.
\label{eq:min-max-eigen-subspace-fem}
\end{eqnarray}

\medskip

Define the projection $P_h: V(h) \rightarrow V_h$ such that for any $u\in V(h)$,
\begin{equation}
  \label{eq:def-projection}
\widehat{a}(u-P_hu,v_h)=0\quad\quad\forall v_h\in V_h.
\end{equation}
{Note that $P_h$ is well defined because $\widehat{a}(\cdot,\cdot)$ is an inner product on  $V_h$ and $\widehat{a}(u,\cdot)$ is a bounded linear functional on $V_h$.}
Furthermore, it can easily be confirmed that $P_h^2=P_h$, i.e., $P_h$ is a projection.
According to \eqref{eq:def-projection}, we have a decomposition of $u\in V(h)$ with respect to the norm $\|\cdot\|_a$: 
\begin{equation}
	\label{eq:orthogonal-decomposition-wrt-a-norm}
\|P_h u \|_a^2+\|(I-P_h)u\|_a^2=\|u\|_a^2~.
\end{equation} 
Unlike \cite{Liu-2015}, we do not assume that $P_h$ is an orthogonal projection on $V(h)$; in general it need not be, since $\widehat a$ may fail to be an inner product on $V(h)$.

\medskip

Let us provide the theorem of explicit eigenvalue bounds using the projection $P_h$.

\begin{theorem}\label{thm:eigenvalue-explicit-bound} 
Suppose that the following relation holds for the projection error estimate of $P_h$:
\begin{equation}
\label{eq:projection-error-est}
  \|\gamma (I-P_h)v\|_b \le C_h \|(I-P_h)v\|_a \quad \forall v \in V(h)\:,
\end{equation}
with a certain computable quantity $C_h$. 
An explicit upper bound for $\mu_{k}$ is given by
\begin{equation}
  \label{eq:mu-upper-bound}
\mu_{k}\leq \mu_{k, h}+C_h^{2} , \quad k=1,2, \cdots, d.
\end{equation}
For $\mu_k\neq 0$, the above estimate yields the following lower bound for $\lambda_k$:
\begin{equation}
  \label{eq:lambda-lower-bound}
  \lambda_{k}\geq \frac{\lambda_{k, h}}{1+ C_h^{2} \lambda_{k,h}}, \quad k=1,2,\cdots, {d'}\:.
\end{equation}
\end{theorem}
\begin{proof}
Let us start with the case that there exists a nonzero $v\in E_{k}(\subset V)$ such that $P_h v=0$. Then the error estimate for $P_h$ implies that 
$$
\|\gamma v \|_b = \|\gamma (I-P_h)v \|_b \le C_h \|(I-P_h)v \|_a = C_h \|v \|_a, \mbox{ i.e., } R(v) \le C_h^2.
$$
Thus, property (\ref{eq:min-max-eigen-subspace}) yields that $\mu_k \le  R(v) \le C_h^2\le \mu_{k,h}+C_h^2$. 

\medskip

Now  consider the case that for any nonzero $v\in E_k$, $P_hv\neq 0$, which implies $\mbox{dim}(P_h(E_k))=k$. 
With the special choice $H_{k,h}:=P_h(E_k)$, the characterization of $\mu_{k,h}$  in \eqref{eq:min-max-eig-subspace-fem} leads to
$$
 \min_{v_h \in P_h (E_k)\setminus\{0\}} R(v_h) \le \mu_{k,h}\:.
$$
Suppose that $v_h=P_h \phi$ ($\phi \in E_k$)  minimizes $R$ in $P_h (E_k)$.  Therefore, 
$$
\|\gamma P_h \phi \|_b \le  \sqrt{\mu_{k,h}} \|P_h \phi \|_a\:.
 $$
 By further utilizing the error estimate for $P_h$ and the decomposition of $\phi$ with respect to the $\|\cdot\|_a$ norm: $\|P_h \phi \|_a^2+\|(I-P_h)\phi\|_a^2=\|\phi\|_a^2$, we have 
 \begin{align*}
 \|\gamma  \phi\|_b 
 & \le  \|\gamma  P_h \phi\|_b + \|\gamma (I-P_h)\phi \|_b  \\
 & \le \sqrt{ \mu_{k,h}} ~ \|P_h \phi\|_a + C_h \|(I-P_h)\phi\|_a \\
 & \le \sqrt{\mu_{k,h}+C_h^2}~  \|\phi\|_a\:.
 \end{align*}
 Since $\phi \in E_k$, {the property of $\mu_k$ in  (\ref{eq:min-max-eigen-subspace})} gives the estimate $\mu_k \le R(\phi) \le \mu_{k,h}+C_h^2$.  

\quad
\end{proof}

\medskip

\begin{remark}
\label{remark:direct-lower-bound}
We highlight two features of  \cref{thm:eigenvalue-explicit-bound} and its proof. 
In Theorem 4.3 of \cite{Liu-2015}, it is assumed that $\widehat{a}(\cdot, \cdot)$ is positive definite over $V(h)$. 
The current theorem eliminates the need for this assumption, making it applicable to more complicated eigenvalue problems, such as the Maxwell eigenvalue problem, where the argument about the positive definiteness of $\widehat{a}(\cdot, \cdot)$ on $V(h)$ is delicate.
Another feature is that, by working with $\mu$ rather than $\lambda$, the proof here avoids the complicated discussion on the space decomposition $V = \mbox{Ker}(\mathcal{K} ) \oplus  \mbox{Ker}(\mathcal{K} )^\perp$, which has been used in \cite{You-Xie-Liu-2019} to handle the lack of positive  definiteness of $b(\gamma\cdot, \gamma\cdot)$ appearing in the Rayleigh quotient for $\lambda$.
\end{remark}

\medskip

\subsection{Application of projection-based eigenvalue bounds}
\label{section-steklov}

The lower eigenvalue bounds provided by Theorem \ref{thm:eigenvalue-explicit-bound} have been applied to various differential operators. 
In this section, as a preparation for high-precision eigenvalue bounds, we recall several known results for two model eigenvalue problems: the Laplacian eigenvalue problem and the Steklov eigenvalue problem.

The analysis is undertaken within the framework of Sobolev spaces.
Let $\Omega$ be a bounded domain in $\mathbb{R}^{k}$  with Lipschitz boundary. 
The $L^2(\Omega)$ function space is the set of real square-integrable functions on $\Omega$,
with  inner product denoted by $(\cdot, \cdot)_{\Omega}$ or $(\cdot, \cdot)$.
We shall use the standard notation for the Sobolev spaces $W^{k,p}(\Omega)$ and their
associated norms $\|\cdot\|_{k,p,\Omega}$ and seminorms $|\cdot|_{k,p,\Omega}$
(see, e.g., \cite{adams2003sobolev}).
For $p=2$, define
$H^k(\Omega)=W^{k,2}(\Omega)$, $\|\cdot \|_{k,\Omega} = \|\cdot\|_{k,2,\Omega}$, 
$|\cdot|_{k,\Omega} = |\cdot|_{k,2,\Omega}$.

For convenience of error analysis and computations, 
from now on we assume $\Omega\subset\mathbb{R}^2$ is bounded and polygonal.  Let $\mathcal{T}^h$ be a triangular subdivision of $\Omega$.
Given an element $K \in \mathcal{T}^h$, $h_K$ denotes the length of the longest edge of $K$. 
Let $h_{max}:=\max_{K\in \mathcal{T}^h} h_K$. That is, $h_{max}$ 
is the maximal edge length of the triangulation. 
Let us introduce the Crouzeix--Raviart finite element space ${V}_h^{\CR}$ defined on $\mathcal{T}^h$:
\begin{equation}
\label{fem space}
\begin{split}
{V}_h^{\CR} := & \{  v \:|\: v \mbox{ is a linear function on each $K\in \mathcal{T}^h$ and continuous at }\\
 & \mspace{28mu} \mbox{the midpoints of  interior edges}\}.
\end{split}
\end{equation}
Since ${V}_h^{\CR} \not\subset H^1(\Omega)$, we introduce the discrete gradient operator $\nabla_h$, which takes the gradient element-wise for $v_h\in {V}_h^{\CR}$. For simplicity, 
$\| \nabla_h v_h \|_{L^2(\Omega)}$ is still denoted by $\| \nabla v_h \|_{L^2(\Omega)}$ or $|v_h|_{1,\Omega}$.

\subsubsection{Laplacian eigenvalue problem}

We consider the Laplacian eigenvalue problem associated with the homogeneous Dirichlet boundary condition:

\medskip

Find $u\in H_0^1(\Omega)\setminus \{0\}$ and $\lambda \in \mathbb{R}$ such that
\begin{equation}
\label{eq:laplacian-evp}
(\nabla u, \nabla v)_\Omega=\lambda (u,v)_\Omega\quad \forall v\in H_0^1(\Omega)~.
\end{equation}

To apply Theorem \ref{thm:eigenvalue-explicit-bound}, we take the following setting for \eqref{eq:laplacian-evp} (see Remark \ref{re:evp-examples}):
$$
  \left\{
    \begin{array}{l} 
      \displaystyle V=H^1_0(\Omega), ~W=L^2(\Omega),~
      \displaystyle a(u,v)=(\nabla u, \nabla v)_\Omega, ~ b(u,v)=(u,v)_\Omega, \\      
\gamma : \mbox{ the compact embedding operator $H^1_0(\Omega)\hookrightarrow L^2(\Omega)$}.  
    \end{array}
  \right.
$$
Define $V_h$ by
$$
V_h=\{v_h \in V_h^{\CR}~|~ v_h=0 \text{ at midpoints of boundary edges of }\mathcal{T}^h \}\:.
$$
Define ${\gamma}_h: V_h \to L^2(\Omega)$ as the embedding operator, 
and define $\hat{\gamma} : V(h)\to L^2(\Omega)$ via 
$\hat{\gamma} (v+v_h):=\gamma v + \gamma_h v_h$ $(v\in V, v_h \in V_h)$ \footnote{Although the decomposition $v+v_h$ is not unique, the definition is well-posed since $\widehat{\gamma} v = \gamma v = \gamma_h v$ for $v\in V\cap {V_h}$.}.

Since $V_h$ is finite dimensional,  $\gamma_h$ is compact. Hence, 
using the compactness of $\gamma$, we have the compactness of $\hat{\gamma}$.  
Note that $v_h$ may be discontinuous across interior edges and does not necessarily vanish on $\partial \Omega$.
The bilinear form $\widehat{a}(\cdot,\cdot)$ for ${V}(h)$ is defined by
\begin{equation}
\widehat{a}(u_h,v_h) := \sum_{K\in \mathcal{T}^h}\int_{K} 
\nabla u_h \cdot \nabla v_h \,\ud x
                          \quad \forall\, u_h, v_h\in {V}(h)\:.
\end{equation}
For any $u,v\in V$,  we have $\widehat{a}(u,v)=a(u,v)$. 

\paragraph{Explicit lower bounds for Laplacian eigenvalues}

Let $\lambda_{k,h}^{\CR}$ be the $k$-th discretized eigenvalue given by the Crouzeix--Raviart FEM. Then, the following lower bound is given by Theorem \ref{thm:eigenvalue-explicit-bound}.
\begin{equation}
	\label{eq:lower-bound-for-laplacian}
\lambda_k \ge \frac{\lambda_{k,h}^{\CR}}{1 + C_h^2\lambda_{k,h}^{\CR}}\quad (1\le k \le \text{dim}(V_h))~.
\end{equation}
Here, for the triangulation of a 2D domain, $C_h$ can be chosen as $C_h=0.1893h_{max}$ \cite{Liu-2015}.
If we would use conforming FEMs instead of  the Crouzeix--Raviart FEM, the computation of $C_h$ would be more complicated; see \cite{Liu-Oishi-2013}.
It is worth pointing out that to have a guaranteed lower bound of $\lambda_k$ using \cref{eq:lower-bound-for-laplacian}, one has to rigorously solve the eigenvalue problem of the discretized matrices from the Crouzeix--Raviart FEM to obtain an exact lower bound of $\lambda_{k,h}^{\CR}$. 
Note that matrix eigenvalue enclosures are accessible by, for example, the algorithm proposed in \cite{Behnke-1991}.

\subsubsection{Steklov eigenvalue problem}
A specific feature here is that the bilinear  form $b(\cdot,\cdot)$ induces a nonzero kernel of $\mathcal{K}=G\gamma$.

\paragraph{The Steklov eigenvalue problem} Find $\lambda >0$ and $u \in H^1(\Omega)$, $\gamma u \neq 0$ such that
\begin{equation}\label{eq:steklov_evp}
(\nabla u, \nabla v)_\Omega + (u,v)_\Omega=\lambda(\gamma u,\gamma v)_{\partial\Omega} \quad \forall\, v\in H^1(\Omega).
\end{equation}
Here, $\gamma:H^1(\Omega)\mapsto L^2(\partial\Omega)$ is the trace operator.
Since $\Omega$ is a bounded Lipschitz domain, $\gamma$ is a compact operator; see, e.g., \cite{Demengel-2012}.

To apply Theorem \ref{thm:eigenvalue-explicit-bound}, we take the following setting (see Remark \ref{re:evp-examples})
\begin{equation*}
	\label{eq:setting-for-steklov}
V:=H^1(\Omega), ~W:=L^2(\partial\Omega),~ 
{a(u,v)} := (\nabla u, \nabla v)_\Omega + (u,v)_\Omega,~
{b(u,v)} := (u,v)_{\partial\Omega}.
\end{equation*}

For the discretized eigenvalue problem, take  $V_h:=V_h^{\CR}$ and ${V}(h):=V+{V}_h$. Let ${\gamma_h}:V_h\to W$ be the piecewise trace operator defined for functions on the boundary elements of $\mathcal{T}^h$. For $v+v_h \in V(h)$ with $v\in V$ and $v_h\in V_h$, $\widehat{\gamma} (v+v_h):=\gamma v + \gamma_h v_h$.\footnote{See the footnote for the case of Laplacian eigenvalue problem.}

Define
$\widehat{a}(\cdot,\cdot)$ on  ${V}(h)$ by
\begin{equation}
\widehat{a}(u_h,v_h) := \sum_{K\in \mathcal{T}^h}\int_{K}\nabla u_h \cdot \nabla v_h +  u_h  v_h\,\ud x
                          \quad \forall\, u_h, v_h\in {V}(h)\:.
\end{equation}
It holds that for $u,v\in V$, $\widehat{a}(u,v)=a(u,v)$. 

\medskip

Let $P_h:V(h) \mapsto {V}_h$ be the  projection as defined in \cref{eq:def-projection}. 
To apply \cref{thm:eigenvalue-explicit-bound}, the following projection error estimate is required.
\begin{equation*}
\label{main_result_ch}
\|\gamma (u-P_h u)\|_{L^{2}(\partial \Omega)} \leq C_h \|u-P_h u\|_{1,\Omega},\,\,\forall u \in {V}\:.
\end{equation*}
In \cite{You-Xie-Liu-2019}, by using the computed value of $\lambda_{1,h}$ from \eqref{eq:eig-problem-m-n-h}, an explicit value of the quantity $C_h$  is provided as
\begin{equation}
\label{eq:C_h_definition_steklov}
C_h := 0.6711
        \max_{K\in \mathcal{T}_b^h}\frac{h_K}{\sqrt{H_K}}
      + \frac{0.1893}{\sqrt{\lambda_{1,h}}}\max_{K\in \mathcal{T}^h} h_K ~, 
\end{equation}
{under the assumption that all  elements of $\mathcal{T}^{h}_b$ have exactly one edge on $\partial \Omega$.}
Here, $\mathcal{T}_b^h$  is the set of elements of $\mathcal{T}^h$ having an edge on $\partial \Omega$;  for $K \in \mathcal{T}_b^h$, $H_K$ denotes the height of $K$ with respect to its boundary edge.

It is worth pointing out that in \cref{eq:C_h_definition_steklov}, except for the term $\lambda_{1,h}$ determined by the eigenvalue problem on the domain, the other terms are only related to the mesh and are essentially independent of the shape of the domain.
{Note that, again, conforming FEMs instead of the
Crouzeix--Raviart FEM would complicate the computation of $C_h$; see \cite{NakanoLiYueLiu+2023}.}

\medskip

\paragraph{Explicit lower bounds for Steklov eigenvalues}
Let $\lambda_{k,h}^{\CR}$ be the $k$-th discretized eigenvalues given by the Crouzeix--Raviart FEM in $V_h$.
With \cref{thm:eigenvalue-explicit-bound} and the explicit projection error estimate for $P_h$, 
we obtain explicit lower bounds for Steklov eigenvalues of \cref{eq:steklov_evp},
\begin{equation}
\label{eq:lower-bound-for-steklov}
\lambda_k \ge  \frac{\lambda_{k,h}^{\CR}}{1+C_h^2 \lambda_{k,h}^{\CR}} ~~( k=1,2,\cdots, {d'})\:.
\end{equation}
Here, ${d'}
=\text{dim}(V_h) - \text{dim}(\text{Ker}(\gamma_h))$, 
$\text{Ker}(\gamma_h) = \{v_h \in {V}_h \,|\, \gamma_h v_h =0\}$.

\begin{remark}
Because $C_h$ (for nonconforming but also for conforming FEMs) is governed by the global maximal edge length of the triangulation, the projection-based bounds \eqref{eq:lower-bound-for-laplacian} and \eqref{main-result-summary} work well on nonuniform meshes with moderate grading, but their accuracy deteriorates on strongly graded meshes where element sizes vary greatly (i.e., very small and very large elements coexist).
This is one motivation for us to introduce our two-stage method.	
\end{remark}
\begin{remark}
\label{remark:sub-optimal-rate-steklov}
Note that since $C_h=O(\sqrt{h})$ as $h\to 0$, the lower eigenvalue bound obtained in
\cref{eq:lower-bound-for-steklov} only converges at rate $O(h)$. This is not
optimal when compared with the approximate eigenvalues themselves, which converges at rate $O(h^2)$ for $H^{2}$-regular solutions; see, e.g., \cite{Yang2010}.
\end{remark}

\section{Second stage: high-precision eigenvalue bounds using Lehmann--Goerisch's theorem}
\label{section:second_stage}
So far, in the first stage of studying eigenvalue bounds, we have successfully constructed projection-based  lower eigenvalue bounds by utilizing for example the nonconforming Crouzeix--Raviart element space.
To provide sharp eigenvalue bounds, a natural idea is to refine the mesh or increase the degree of the polynomials used in the finite element spaces.
For nonconforming FEMs, one can refine the mesh to improve the precision, but a high-order Crouzeix--Raviart-like FEM is generally not available.

Instead, in the second stage of our proposed algorithm to obtain high-precision eigenvalue bounds, let us introduce Lehmann--Goerisch's theorem using high-order conforming FEMs and the rough eigenvalue bounds obtained in the first stage.

The idea of adopting Lehmann--Goerisch's theorem and  conforming FEMs has also been successfully used to bound
the Poincar{\'{e}} constants, where the high-order conforming FEM produces quite efficient eigenvalue bounds (see \cite{liu-2013-2}), and to give computer-assisted existence and multiplicity proofs for various non-linear boundary value problems (see, e.g.,  \cite{plum1993numerical,breuer2003multiple,breuer2006computer,nakao2019numerical_book});

\subsection{Lehmann--Goerisch's theorem}

Let us introduce the formulation of the eigenvalue problem to be investigated by applying the Lehmann--Goerisch theorem.

\medskip

Find $u \in D$ and $\lambda \in \mathbb{R}$ such that,
\begin{equation}
\label{eq:evp-lg-method}
M(u,v) = \lambda N(u,v)\quad \forall v\in  D~.
\end{equation}
The notation follows the original one used by F. Goerisch. Compared with the eigenvalue problem defined in \cref{eq:objective-evp}, we have 
$$
D=V, \quad M(\cdot, \cdot)=a(\cdot, \cdot), \quad  N(\cdot, \cdot)=b(\gamma \cdot, \gamma \cdot)~.
$$ 

Lehmann--Goerisch's theorem first declares the existence of eigenvalues in a certain range. By further utilizing rough eigenvalue bounds, which can be obtained from the first stage of our algorithm, one can obtain high-precision bounds for eigenvalues of \cref{eq:evp-lg-method}. 

\medskip

Now, let us display the assumptions made for Lehmann--Goerisch's theorem \footnote{The assumptions here are the same as in  \cite{goerisch1985eigenwertschranken, goerisch1986convergence}. 
Note that the assumptions in \cite{goerisch1985eigenwertschranken} are formulated with operators, while in a later paper of Goerisch and Albrecht \cite{goerisch1986convergence}, the eigenvalue problem and the assumptions are formulated with  bilinear forms.}.

\medskip

{\bf Assumption:}
\begin{enumerate}
 \item [A1]
  Let $D$ be a real vector space, $M(\cdot, \cdot)$ and $N(\cdot,\cdot)$ be symmetric bilinear forms on  $D$. $M(\cdot, \cdot)$ is positive definite on $D$.
 \item[A2] There exists a sequence $\{ (\lambda_i, \varphi_i)\}_{i\in \mathbb{N}}$ in $\mathbb{R}\times D$ of eigenpairs satisfying 
 $$
 M(\varphi_i, v) =\lambda_i  N(\varphi_i,v)\quad  \forall v\in D;\quad
M(\varphi_i,\varphi_j )=\delta_{ij} \quad \text{for } i,j\in\mathbb{N}~.
 $$
 Here, $\delta_{ij}$ is the Kronecker delta.
 Moreover, 
 \begin{equation}
N(v,v) = \sum\limits_{i=1}^\infty \lambda_i\left| N(v,\varphi_i) \right|^2 \quad \forall v\in D~.
 \end{equation}

\item[A3] ~Let $X$ be a real vector space and $b_G$ be a symmetric positive semidefinite bilinear form on $X$. Let $T:D\rightarrow X$ be a linear operator such that
\begin{equation*}	
b_G(Tu,Tv)=M(u, v)\quad\forall u,v\in D~.
\end{equation*}
\item [A4] ~Given   $v_1, v_2, \cdots, v_n\in D$, let $w_1, w_2, \cdots, w_n \in X$ satisfy
\begin{equation}
	\label{eq:Lehmann--Goerisch-v-to-w}
b_G(w_i, T v)=N(v_i,v)\quad \forall v  \in D \quad (i=1,\dots,n)~.
\end{equation}
\item [A5] Given $\rho\in\mathbb{R}$, $\rho>0$,
define $n\times n$ matrices $A_0, A_1, A_2, A, B$ by
$$
A_0 \!:= \big(M(v_i,v_j)\big)_{i,j=1}^n, \,
A_1 \!:= \big(N(v_i,v_j)\big)_{i,j=1}^n, \,
A_2 \!:= \big(b_G(w_i,w_j)\big)_{i,j=1}^n, \, 
$$
\begin{equation}
\label{eq:A-B-in-LH-Thm}
A := A_0-\rho A_1, ~~ B := A_0-2\rho A_1 + \rho^2 A_2.
\end{equation}

Suppose $B$ is positive definite. Let $\nu_1 \le \nu_2 \le \cdots \le \nu_n$ be the eigenvalues of $Az=\nu Bz$. 
Let $q$ be the number of negative eigenvalues among them. 
\end{enumerate}

\medskip

Now, we quote the result from Lehmann--Goerisch's theorem.

\begin{theorem}[Lehmann--Goerisch's theorem] 
\label{thm:Lehmann--Goerisch}
Under {assumptions A1-A5}, for  each $k$ $(1\le k \le q)$, 
the interval $\left[\rho -\rho/(1-\nu_k),\rho\right)$ contains at least $k$ eigenvalues of  
   $\{\lambda_i\}_{i\in \mathbb{N}}$, counting by multiplicity.
\end{theorem}

\begin{lemma}
\label{lem:Lehmann}
Given $v \in D$, let  $\tilde{u}  \in D$ be the solution to the equation: 
\begin{equation}
	\label{eq:lehmann-setting}
 M(\tilde{u}, \phi)=N(v,\phi)\quad \forall \phi  \in D ~.
\end{equation}
For any $w \in X$  satisfying
\begin{equation*}
b_G(w, T \phi )=N(v,\phi)\quad \forall \phi  \in D ~,
\end{equation*}
we have $b_G(w, w) \ge b_G(T\tilde{u}, T\tilde{u})=M(\tilde{u}, \tilde{u})$.
\end{lemma}
\begin{proof}
Let $\widehat{w}:=w-T\tilde{u} \in X$.  Since  
$$
b_G(\widehat{w}, T \phi )=
 b_G(w, T \phi ) - M(\tilde{u}, \phi) = N(v,\phi)-N(v,\phi)=0 \quad \forall \phi  \in D\:,$$ 
we have
$$
b_G(w, w)= b_G(T\tilde{u}+\widehat{w}, T\tilde{u}+\widehat{w})
=b_G(T\tilde{u}, T\tilde{u} )+ b_G(\widehat{w}, \widehat{w}) \ge b_G(T\tilde{u}, T\tilde{u}).
$$

\end{proof}
\begin{remark}
{(Selection of the $w_i$ term)}
\label{remark:selection_of_w_i}
Generally, there is freedom to select $w_i$ under the condition \eqref{eq:Lehmann--Goerisch-v-to-w}. A selection of $w_i$'s with  ``small" values of $b_G(w_i,w_i)$ will produce ``good" lower eigenvalue bounds. 
In fact, $b_G(w_i,w_i)$ under the constraint condition \eqref{eq:Lehmann--Goerisch-v-to-w} is minimized by $w_i = T \widetilde{u}_i$, where $\widetilde{u}_i$ is obtained from the standard setting of Lehmann's method by solving \eqref{eq:lehmann-setting}.
However, this original Lehmann's equation is generally difficult (or impossible)  to solve explicitly.
The introduction of the $w_i$ term  by Goerisch (called ``Goerisch's XbT-concept"; see, e.g., \cite[Lemma 10.25]{nakao2019numerical_book})  makes it flexible to construct $w_i$ to obtain eigenvalue bounds for the target eigenvalue problem. In practical computing, we construct a finite dimensional space $X_h(\subset X)$ and compute the $w_i$ term by solving the minimization problem (see  \cite[Remark 5]{liu-2013-2} and  \cite[Remark 10.26 (c)]{nakao2019numerical_book}):
\begin{equation}
\label{eq:minimization-problem-in-LG}
	\min_{w_i \in X_h} b_G(w_i, w_i) \text{ subject to condition \eqref{eq:Lehmann--Goerisch-v-to-w}}\:.  
\end{equation}
To have a rigorous eigenvalue bound, the  minimization problem \cref{eq:minimization-problem-in-LG}	 does nevertheless not need to be solved exactly, while the condition \eqref{eq:Lehmann--Goerisch-v-to-w} has to be satisfied strictly.
\end{remark}

\bigskip

From now on, suppose $A_1$ is positive definite. 
Let $\Lambda_n$ be the maximal eigenvalue of $A_0x=\Lambda A_1 x$.

Effort has been paid to validate the positive definiteness of matrix $B$ in assumption A5, see, e.g., 
\cite[Theorem 10.10, 10.21, 10.29]{nakao2019numerical_book}.
The following Lemma provides a direct condition for this aim.

{
\begin{lemma}
\label{lem:cond-for-positive-B}
In the case where $\rho > \Lambda_n$, the matrix $B$ in equation \eqref{eq:A-B-in-LH-Thm} is indeed positive definite, as required in assumption A5.
\end{lemma}
\begin{proof}
Let $v$ ($\neq 0$) be a linear combination of $v_1, v_2, \ldots, v_n \in D$. 
Let $\tilde{u}$ be the solution of equation \eqref{eq:lehmann-setting} with $v_i$ replaced by $v$. That is,
\begin{equation}
	\label{eq:local_B_positive_1}
M(\tilde{u}, \phi)=N(v,\phi), ~\forall \phi \in D.
\end{equation}
Let $\widehat{w} := T\tilde{u}$. We have $b_G(\widehat{w}, Tv)= b_G(T\tilde{u}, Tv)=
M(\tilde{u}, v)=N(v,v)$, and 
\[
0 \leq b_G(Tv - \rho \widehat{w}, Tv - \rho \widehat{w}) = M(v,v) - 2\rho N(v,v) + \rho^2 b_G(\widehat{w}, \widehat{w}).
\]
Moreover,  $b_G(Tv - \rho \widehat{w}, Tv)  \ne 0$, because otherwise, 
$$
M(v,v)=b_G(Tv, Tv)=b_G(\rho \widehat{w}, Tv)
=\rho N(v,v).
$$ 
However, $M(v,v) - \rho N(v,v) =0$  contradicts the assumption that $\rho > \Lambda_n$.

From the confirmed  fact that $b_G(Tv - \rho \widehat{w}, Tv)  \ne 0$, we have
\footnote{Note that, for any symmetric and positive semidefinite bilinear form $b_G(\cdot, \cdot)$,  $b_G(f,f)=0$ implies $b_G(f,g)=0$ for all $g\in X$ due to the Cauchy--Schwarz inequality.}
\begin{equation}
	\label{eq:local_2}
b_G(Tv - \rho \widehat{w}, Tv - \rho \widehat{w}) > 0.
\end{equation}
Given any $w\in X$   such that
\[
b_G(w, T\phi) = N(v, \phi) \quad \forall \phi \in D,
\]
we have 
 $b_G(w, w) \geq b_G(\widehat{w}, \widehat{w})$ by Lemma \ref{lem:Lehmann}. 
By further using the inequality \eqref{eq:local_2}, 
\begin{align*}
& M(v,v) - 2\rho N(v,v) + \rho^2 b_G(w, w) \ge M(v,v) - 2\rho N(v,v) + \rho^2 b_G(\widehat{w}, \widehat{w}) &\\
& = b_G(Tv - \rho \widehat{w}, Tv - \rho \widehat{w}) > 0.&	
\end{align*}
This implies that the matrix $B$ is positive definite.
\end{proof}

 }
 
 \medskip
 
 The following theorem gives an application of Theorem \ref{thm:Lehmann--Goerisch}  to obtaining lower  bounds for $n$ eigenvalues: $\lambda_{m-n+1}, \cdots,\lambda_{m}$, by using a lower bound (denoted by $\rho$)  of $\lambda_{m+1}$; see Figure \ref{fi:eigenvalues-for-LG-method} for an illustration  of the involved eigenvalues.

\begin{figure}[h]
\begin{center}
\includegraphics[width=10cm]{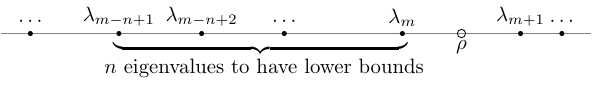}
\end{center}
\caption{Bounding eigenvalues by using Lehmann--Goerisch's theorem.}
\label{fi:eigenvalues-for-LG-method}
\end{figure}

\medskip

\begin{theorem}
\label{theorem:lower_bound_for_first_n_eigenvalues}
Suppose $\rho$ is a lower bound to the  $(m+1)$th eigenvalue $\lambda_{m+1}$ for some $m \ge n$ ($m\in \mathbb{N}$), i.e., 
\begin{equation}
	\label{eq:condition_for_rho}
\rho \le \lambda_{m+1}~.
\end{equation}
If assumptions A1-A5 are satisfied, then the following  lower eigenvalue bound holds.
\begin{equation}\label{Lehmann--Goerisch bounds}
\lambda_{m+1-k} \geq \rho - \frac{\rho}{1 -\nu_k} \quad (1\leq k\leq q)\:.
\end{equation}
In particular, if $\rho>\Lambda_n$,
then $\nu_n<0$, i.e., $q=n$, and  
\begin{equation}\label{Lehmann--Goerisch bounds-2}
\lambda_{m+1-k} \geq \rho - \frac{\rho}{1 -\nu_k} \quad (1\leq k\leq n)\:.
\end{equation}

\end{theorem}

\medskip

We  remark that in the case that $m=n$, under  the essential condition  (see also \cite[Theorem 10.31]{nakao2019numerical_book})
\begin{equation}
\label{eq:cond_for_rho_m_eq_n}
\Lambda_n < \rho \le \lambda_{n+1}\:, 	
\end{equation}
the right-hand side of \eqref{Lehmann--Goerisch bounds-2}  
provides lower  bounds for $\lambda_1,\cdots, \lambda_n$.

\begin{remark}
\label{remark:rho_condition}
It is crucial to point out that in our proposed two-stage algorithm, the projection-based lower eigenvalue bounds in the first stage  
serve to compute $\rho$ satisfying 
\eqref{eq:condition_for_rho} or \eqref{eq:cond_for_rho_m_eq_n}.
\end{remark}

\medskip

\begin{remark} The lower bound given by Lehmann--Goerisch's theorem is monotonically increasing with respect to the {\rm a priori} lower bound $\rho~ (\le \lambda_{m+1})$ (Theorem 4.1 of \cite{Beattie+Goerisch1995}).
In practice, even with a rough $\rho$,
the eigenvalue bounds in \cref{Lehmann--Goerisch bounds} are still sharp if
$v_i$ is a good approximation to the exact eigenfunction.
\end{remark}

\subsection{Application of Lehmann--Goerisch's theorem to concrete eigenvalue problems}
    In this section, we show the application of the Lehmann--Goerisch theorem to two concrete eigenvalue problems: the Dirichlet Laplacian eigenvalue problem and the Steklov eigenvalue problem. Here, to  simplify the argument, we consider  lower bounds for the first $n$ eigenvalues. Generally, the application of Lehmann--Goerisch's theorem does not need to start from the first eigenvalue.

As a feature of Lehmann--Goerisch's method, it takes advantage of  graded meshes (i.e., meshes with refinement around possible singular parts) to recover the convergence rate of the approximate solution, when the objective solution has singularities.
 Different from the lower eigenvalue bound of Theorem \ref{thm:eigenvalue-explicit-bound} based on interpolation or projection-based error  estimation, the lower eigenvalue bound from Lehmann--Goerisch's theorem is not affected by the size of rough  elements when graded meshes are used.
 In the last section on numerical computations, an example using a graded mesh for the L-shaped domain is provided to demonstrate this feature.
 
 \medskip
	     
As a preparation for the implementation of Lehmann--Goerisch's theorem, let us introduce several FEM spaces for the currently considered second-order eigenvalue problems. Note that here the degree of a polynomial is defined by the summation of exponents for each variable. For example, $xy$ has the degree $2$. The order $p$ of the following FEM spaces will be taken as $p\ge 1$.
\begin{itemize}
\item Let $\text{CG}^p_h$ be the Lagrange FEM space consisting of continuous piecewise polynomials of degree up to $p$.
	\item Let $\text{RT}^p_h(\subset H(\Div ; \Omega))$ be the Raviart--Thomas FEM space of degree $p$:
\begin{align*}
\text{RT}^p_h := \Big\{v~:~
& v= \left(\begin{array}{c}\mspace{-10mu}c_1 \mspace{-10mu}\\ \mspace{-10mu}c_2 \mspace{-10mu}\end{array}\right)
+ c_3 \left(\begin{array}{c}\mspace{-10mu}x \mspace{-10mu}\\ \mspace{-10mu}y \mspace{-10mu}\end{array}\right) 
     \,\,\text{ on each element $K$, where the $c_i$'s} \\
    & \text{are polynomials on $K$ of degree up to $p$;} \\
& \mspace{-10mu}\,\, v\cdot\vec{n} \,\,\text{is continuous across each interior edge } e. ~\Big\}~.
\end{align*}
\item Let $\text{DG}_h^p$ be the space of piecewise polynomials with degree up to $p$. 
The property of the Raviart--Thomas space implies that
$\Div(\text{RT}_h^p) = \text{DG}_h^p$. 
\end{itemize}
\subsection{The Laplacian eigenvalue problem}

We consider the eigenvalue problem defined in \eqref{eq:laplacian-evp}, and aim to obtain high-precision eigenvalue bounds for $\lambda_1, \cdots, \lambda_n$. It is assumed that $\lambda_n$ and $\lambda_{n+1}$ are well separated from each other.
It is worth pointing out that to apply Theorem \ref{theorem:lower_bound_for_first_n_eigenvalues} for obtaining the lower bounds for the first $n$ eigenvalues, the projection-error based lower bound from the first 
stage is important and will be used in the selection of $\rho$ such that $\Lambda_n < \rho \le \lambda_{n+1}$ (see Remark \ref{remark:rho_condition}).

Introduce 
the following setting for  Lehmann--Goerisch's theorem.
$$
D:=H^1_0(\Omega); ~
M(u,v) :=(\nabla u, \nabla v)_\Omega,~
N(u,v):=(u, v)_\Omega\quad \forall u,v \in D\:.
$$

Goerisch's $XbT$ setting is taken as 
$$
X:=\left(L^2(\Omega)\right)^2; \quad 
b_G(u,v) := (u,v)_\Omega \quad \forall u,v\in X~;\quad T:=\nabla~.
$$
Take the eigenpairs $\{\lambda_i, u_i\}_{i=1}^\infty$ of \eqref{eq:laplacian-evp} for the ones 
in Assumption A2.
It is easy to confirm Assumption A3, i.e., $b_G(\nabla u, \nabla v)=M(u, v)$ for all $u,v\in {D}$.

\medskip

Now, let us describe the selection of functions required by Lehmann--Goerisch's theorem.  
 
\paragraph{Selection of $v_i$}
 Let us consider the discretized eigenvalue problem: Find $u_h(\not = 0)$ in $\text{CG}^p_h$ and 
$\lambda_h \in \mathbb{R}$ s.t.
\begin{equation}
\label{eq:eig-h-laplacian-cg-p}	
M(u_h, v_h) = \lambda_{h} N(u_h, v_h)\quad \forall v_h \in \text{CG}^p_h~.
\end{equation}
Denote the exact eigenpairs for the first $n$ eigenvalues by $\{\lambda_{i,h}, u_{i,h}\}_{i=1}^{n}$. 
The approximate eigenfunctions $v_1, \cdots, v_n$ required in A4 of Lehmann--Goerisch's theorem are selected as $v_i \approx u_{i,h}$. Note that $v_i$ does not need to be exactly equal to $u_{i,h}$, since any conforming approximation is allowed in the setting of  Lehmann--Goerisch's theorem.

\paragraph{Calculation of $w_i$}
{For each  $v_i$ $(\approx u_{i,h})$}, the constraint condition \eqref{eq:Lehmann--Goerisch-v-to-w} for $w_i \in X=(L^2(\Omega))^2$ becomes
\begin{equation}
\label{eq:LG-cond-dirichlet-v2w}
(w_i , \nabla v)= (v_i, v) \quad \forall v\in H^1_0(\Omega)\:.
\end{equation}
By the distributional definition of $\Div $, 
the condition \eqref{eq:LG-cond-dirichlet-v2w} means that $w_i$ belongs to $H(\Div;\Omega)$ and has to satisfy  \footnote{
By taking test functions $v$ from $C_0^\infty(\Omega)$ in \eqref{eq:LG-cond-dirichlet-v2w}, we can see that the distributional $\Div w_i$ is actually in $L^2(\Omega)$,
and $\Div w_i + v_{i} =0$ holds.
} 
\begin{equation}
	\label{eq:LG-cond-dirichlet-v2w-ph}
\Div w_i +v_{i}=0\quad \text{ in } \Omega~.	
\end{equation}

Since $v_{i}\in \text{CG}^p_h(\subset \text{DG}_h^p=\Div (\text{RT}^p_h))$, 
one can find a $w_i\in \text{RT}^p_h$ that solves \eqref{eq:LG-cond-dirichlet-v2w-ph}  exactly. That is, $\Div w_i +v_{i}=0$ holds pointwise.
In numerical computation, equation \eqref{eq:LG-cond-dirichlet-v2w-ph} is asserted by using test functions from $\text{DG}^p_h$, i.e., by requiring
$$
(\Div w_i + v_{i},\psi_h) =0\quad \forall \psi_h \in \text{DG}^p_h\:.
$$
The solution $w_i$ to \eqref{eq:LG-cond-dirichlet-v2w-ph} is not unique.
To have a ``good" selection of $w_i$ to obtain a tight eigenvalue bound, let us follow Remark \ref{remark:selection_of_w_i} to consider the following minimization problem:
\begin{equation}
	\label{eq:min_w_for_v_dirichlet}
\min_{w_i\in \text{RT}_h^p } \|w_i\| \quad \text{subject to } 
\Div w_i+ v_{i} = 0~.
\end{equation}
The saddle point formulation of \eqref{eq:min_w_for_v_dirichlet} is to find $(w_i,\phi_{i,h})  \in \text{RT}_h^p \times \text{DG}^p_h$ s.t.
\begin{equation}
\label{eq:discrete-equ-laplacian}
\left\{
\begin{array}{l}
(w_i, q_h) + (\phi_{i,h},\Div q_h)=0\quad \forall q_h \in \text{RT}_h^p~,\\
(\Div w_i, \psi_h) + (v_{i},\psi_h) =0\quad \forall \psi_h \in \text{DG}^p_h~.
\end{array}
\right.
\end{equation}
The system (\ref{eq:discrete-equ-laplacian}) admits a unique solution $(w_i, \phi_{i,h})$ (c.f., e.g., \cite{Arnold+F.Brezzi1985}, \S IV.1 Prop. 1.1 of \cite{Brezzi+Fortin1991}). As noted in \cref{remark:selection_of_w_i}, to have rigorous eigenvalue bounds, $w_i$ does not necessarily  have to be the exact minimizer of \cref{eq:min_w_for_v_dirichlet}, while the constraint condition \cref{eq:LG-cond-dirichlet-v2w-ph}  must hold exactly. Such a $w_i$ can be obtained as an interval function  by solving \cref{eq:discrete-equ-laplacian} with  interval arithmetic, which however requires extra computing cost compared to an approximate  solver for \cref{eq:discrete-equ-laplacian}. For an efficient way of computation of $w_i$, see Remark \ref{remark:Lehmann--Goerisch-with-shift} for the technique  to introduce a shift to the eigenvalue problem.

\paragraph{Selection of $\rho$}
~
In the first stage, we have applied  the projection-based method in \S 3 to obtain an exact lower bound $\rho$ for eigenvalue $\lambda_{n+1}$, i.e., $\rho \le \lambda_{n+1}$. 
In the case that the Crouzeix--Raviart FEM is used in the eigenvalue lower bound \eqref{eq:lower-bound-for-laplacian}, 
$\rho$ is taken as
$$
\rho:= \frac{\lambda^{\CR}_{n+1,h}}{1+C_h^2 \lambda^{\CR}_{n+1,h}}~.
$$
The matrices $A_1$, $A_1$, $A_2$, $A$ and $B$ in A5 of Lehmann--Goerisch's theorem are constructed by using the above $v_{i}$'s and $w_i$'s $(i=1,\cdots, n)$.
Let $\Lambda_n$ be the maximal eigenvalue of the  eigenvalue problem with the matrices $A_0$ and $A_1$:
\begin{equation}
\label{eq:Lambda-n-LG}	
A_0x=\Lambda A_1 x~.
\end{equation} 
With a proper selection of $n$ such that $\lambda_n$ and $\lambda_{n+1}$ are separated from each other, one can calculate $v_{i}$ using refined meshes and $\rho$ to make sure the following condition holds
$$
(\lambda_n\le )~\Lambda_{n} < \rho \le \lambda_{n+1}~.
$$
Note that the above selection of $\rho$ implies  that $B$ is  positive definite; see Lemma   \ref{lem:cond-for-positive-B}.

\medskip

To have an easy-to-follow description of Lehmann--Goerisch's theorem, let us summarize the computation procedure in Algorithm \ref{alg:Lehmann--Goerisch}.

\begin{algorithm}[h]
\caption{Guaranteed eigenvalue bounds by Lehmann--Goerisch's theorem}
\label{alg:Lehmann--Goerisch}
\begin{normalsize}

{{\bf Objective:} Lower bounds for} the first $n$ eigenvalues. It is assumed that $\lambda_n$ and $\lambda_{n+1}$ are well separated from each other so that the condition \eqref{eq:rho-in-algorithm-desc} can be satisfied.

\medskip

{\bf Procedure:}

\begin{itemize}
\item[Step 1:] {Solve the eigenvalue problem (\ref{eq:eig-h-laplacian-cg-p}) approximately to obtain $v_i (\approx u_{i,h}) \in \text{CG}^p_h $. 
Solve the eigenvalue problem 
\eqref{eq:Lambda-n-LG} rigorously to obtain $\Lambda_n$, which provides an upper bound of $\lambda_{n,h}$.}
	\item [{Step 2:}] Solve the eigenvalue problem (\ref{eq:laplacian-evp}) in $V_h^{\CR}$ { in a verified way by using interval arithmetic} to calculate $\lambda^{\CR}_{n+1,h}$. 
With refined meshes for $V_h^{\CR}$ and $\text{CG}^p_h$, obtain a $\rho$ such that
\begin{equation}
	\label{eq:rho-in-algorithm-desc}
\lambda_{n,h} \le \Lambda_n < \rho:= \frac{\lambda^{\CR}_{n+1,h}{1+C_h^2 \lambda^{\CR}_{n+1,h}}}  \le \lambda_{n+1} ~.
\end{equation}

\item [{Step 3:}]
For each $v_{i}$ ($1\le i \le n$),
let  $w_{i}$ be the solution of (\ref{eq:discrete-equ-laplacian}) that satisfies \cref{eq:LG-cond-dirichlet-v2w-ph} exactly.

\item [{Step 4:}] Define $A_0$, $A_1$, $A_2$, $A$ and $B$ with the functions
$v_{i}$, $w_{i}$ ($1\le i \le n$).
Let the eigenvalues of $Ax=\nu Bx$ be $\nu_1 \le \cdots \le \nu_{n}$. The relation  $\lambda_{n,h}  < \rho$ makes sure that $B$ and $-A$ are positive definite and hence $\nu_n<0$. The lower bound of $\lambda_k$ in (\ref{eq:laplacian-evp}) is given by
$$
\lambda_k \ge \rho - \rho/(1-\nu_{n+1-k})\quad (1\le k \le n)\:.
$$

\end{itemize}
\end{normalsize}
\end{algorithm}

\medskip

\begin{remark}
\label{remark:Lehmann--Goerisch-with-shift}
Let us introduce a well-used technique that avoids heavy computation in rigorously solving the linear system 
\eqref{eq:discrete-equ-laplacian} when seeking the $w_i$ terms.  We did not use this technique for the examples in Section 5, but refer to the technique here for future purposes. Let us consider the eigenvalue problem with a shift  $\widehat{\lambda}(>0)$, the setting of which is given by
$$
D:=H^1_0(\Omega); ~
M(u,v) :=(\nabla u, \nabla v)_\Omega+\widehat{\lambda} (u,v)_\Omega,~
N(u,v):=(u, v)_\Omega~ \forall u,v \in D\:.
$$
Goerisch's $XbT$ setting is taken as $X:=\left(L^2(\Omega)\right)^3$, $Tu:=\{\nabla u, u\}(\in X)$ for $u \in D$. 
Rewrite each $u\in X$ by $u=\{u^{(1)}, u^{(2)}\}$, $u^{(1)}\in \left(L^2(\Omega)\right)^2$, $u^{(2)} \in L^2(\Omega)$. 
For $u,v \in X$, 
$$
b_G({u},{v}) := ({ u^{(1)}},{v^{(1)}})_\Omega  + \widehat{\lambda} ({ u^{(2)}},{v^{(2)}})_\Omega ~.
$$
For each approximate eigenfunction $v_i$, the condition \cref{eq:Lehmann--Goerisch-v-to-w} for $w_i=\{w_i^{(1)},w_i^{(2)}\}\in X$  becomes 
\begin{equation}\label{eq:lg_new_constraint_cond_under_shift}
\Div w_i^{(1)} + \widehat{\lambda}  w_i^{(2)} + v_i=0~\text{ in } \Omega.
\end{equation}
For given $v_i\approx u_{i,h}$, calculate ${w}^{(1)}_i \in \text{RT}_h^{\:p}$ by solving 
\cref{eq:discrete-equ-laplacian} with an approximate  solver. 
Then, calculate $w_i^{(2)}$ as $w_i^{(2)} = (- \Div({w}^{(1)}_i ) - v_{i})/\:\widehat{\lambda}$ in a rigorous way. The resulting $w_i=\{w_i^{(1)},w_i^{(2)}\}$ satisfies the condition \cref{eq:lg_new_constraint_cond_under_shift} exactly. By applying Lehmann--Goerisch's theorem to the current setting, one can directly obtain lower bounds for the shifted eigenvalues, and thus bounds for the original eigenvalues after re-subtracting the shift $\widehat{\lambda}$.

\end{remark} 


\subsection{The Steklov eigenvalue problem}
\label{subsec:Lehmann--Goerisch-setting}
In this subsection, we consider high-precision bounds for the first $n$ eigenvalues for the Steklov eigenvalue problem.
Let us take the following setting for the Lehmann--Goerisch theorem.
\begin{enumerate}
\item [(a)] Let $D = H^1(\Omega)$ and define the bilinear forms
$M(\cdot, \cdot)$ and $N(\cdot, \cdot)$ by
\begin{equation*}
	M(u,v):= \int_\Omega \nabla u \cdot  \nabla v + uv \: \ud x,\quad
N(u,v) := \int_{\partial \Omega} \gamma u \gamma v \:\ud s~~ \forall u,v\in {D}~.
\end{equation*}
Here, $\gamma$ is the trace operator.

	\item [(b)] Let  $\{\lambda_i, u_i\}_{i \in \mathbb{N}}$  be the eigenpairs of the eigenvalue problem  \eqref{eq:eig-problem-with-lambda}. 
   Note that $\{u_i\}_{i \in \mathbb{N}}$ is a complete orthonormal system of $\mbox{Ker}(\mathcal{K})^\perp (\subset H^1(\Omega))$,   where the orthogonal complement is taken with respect to  $M(\cdot, \cdot)$. 
    Given $v\in  D$,  decompose $v$ by $v=v_0+v'$, $v_0 \in H_0^1(\Omega)$,
$v' \in \mbox{Ker}(\mathcal{K})^\perp$ and 
\begin{equation}
\label{eq:local_v_prime_by_ui}
v'=\sum_{i=1}^{\infty}  M(v',u_i) u_i.\end{equation}
Since $v_0 \in H_0^1(\Omega)=\mbox{Ker}(\mathcal{K})$, we have $ M(u_i, v_0)= \lambda_i N(u_i,v_0)=0$. 
Hence,
$$
 M(u_i, v) =
 M(u_i, v') =\lambda_i N(u_i, v')
 =\lambda_i N(u_i, v).
$$
It is clear that $M(u_i,u_j )=\delta_{ij}$ and $N(u_i,u_j )=1/\lambda_i\delta_{ij}$ ($\delta_{ij}$: Kronecker delta). 
Therefore,  using \eqref{eq:local_v_prime_by_ui} and the continuity of $N(\cdot, \cdot)$ with respect to $M(\cdot,\cdot)$,
\begin{align}
	N(v,v)  &= N(v',v') =  \sum_{i=1}^{\infty} N(u_i, u_i ) \cdot  |M(v',u_i)|^2 & \notag \\
&=\sum_{i=1}^\infty \frac{1}{\lambda_i} \cdot \lambda_i^2  |N(v',u_i)|^2 
=\sum_{i=1}^\infty \lambda_i  |N(v',u_i)|^2 = \sum_{i=1}^\infty \lambda_i  |N(v,u_i)|^2 \:.	&
\notag 
\end{align}
With the above argument,  we have confirmed that  $\{\varphi_i\}_{i \in \mathbb{N}}$ in Assumption A2 can be taken as $\{u_i\}_{i \in \mathbb{N}}$.

\item [(c)] Let ${X} = \big(L^2(\Omega)\big)^3$ and define $b_G(\cdot, \cdot)$ 
for	$\phi, \psi \in {X}$ by
$
b_G(\phi,\psi) := \int_\Omega \phi \cdot \psi \, dx \:.
$
Let $Tv := \{\nabla v,v\}$ for $v\in{D}$. Hence, $b_G(Tu,Tv)=M(u,v)$ for all $u,v\in {D}$.

\end{enumerate}

\medskip

Now, let us describe the function  required by  Lehmann--Goerisch's theorem.

\paragraph{Selection of $v_i$}
Let us consider the discretized eigenvalue problem: 
\begin{equation}
M(u_h, v_h) = \lambda_{h} N(u_h, v_h)\quad \forall v_h \in \text{CG}^p_h~.
\end{equation}
Or, conversely, 
\begin{equation}	\label{eq:steklov_eig_h}
N(u_h, v_h) = \mu_{h} M(u_h, v_h)\quad \forall v_h \in \text{CG}^p_h~.
\end{equation}
Take $\mu_{i,h}$ ($i=1,\cdots,n$) as the $n$ largest positive eigenvalues of \eqref{eq:steklov_eig_h} and $u_{i,h}\in \mbox{Ker}(\mathcal{K})^\perp$  ($i=1,\cdots,n$) as the corresponding eigenfunctions. Then, define $\lambda_{i,h}:=1/\mu_{i,h}$ $(i=1,\cdots,n)$.
The functions $v_1, \cdots, v_n$ required by A4 of Lehmann--Goerisch's theorem are selected  in $\text{CG}^p_h$ as approximations  to $u_{1,h}, \cdots, u_{n,h}$. 

\medskip

\paragraph{Calculation of $w_i$}
For each $v_i\approx u_{i,h}$, the equation \cref{eq:Lehmann--Goerisch-v-to-w} to determine ${w}_i=\{w_{i}^{(1)},w_{i}^{(2)}\} \in  X$, $w_i^{(1)} \in (L^2(\Omega))^2$, $w_i^{(2)} \in L^2(\Omega)$ becomes
\begin{equation}\label{sol-operator-G-steklov}
( w_{i}^{(1)}, \nabla v)_{\Omega}
+( w_{i}^{(2)}, v)_{\Omega} = (\gamma v_{i}, \gamma  v)_{\partial\Omega}\quad\forall v \in {D}.
\end{equation}
It is easy to see that $w_{i}^{(1)}$ must belong to $H(\Div ; \Omega)$ and satisfy 
\begin{equation}
\label{eq:beta-equation-bdc}
  - \Div w_{i}^{(1)} + w_{i}^{(2)} =0 \text{ in } \Omega, \quad w_{i}^{(1)}  \cdot \vec{n} = \gamma v_{i} \text{ on } {\partial\Omega}.	
\end{equation}
Therefore, $w_{i}^{(2)}:=\Div w_{i}^{(1)}$ has to be chosen, whence 
$$
b_G(w_i, w_i) =\|w_{i}^{(1)}\|^2_{\Omega} + \|\Div  w_{i}^{(1)}\|^2_{\Omega}~.
$$
Since 
$\{(w_h\cdot \vec{n})|_{\partial\Omega} \::\: w_h\in {RT}^{p}_h\}$ contains all piecewise polynomials of degree $p$ on boundary edges of $\mathcal{T}^h$, 
$w_{i}^{(1)}$ can be selected from ${RT}^{p}_h$ such that the second equation of \cref{eq:beta-equation-bdc} holds.  
As suggested in Remark \ref{remark:selection_of_w_i}, the $w_i$ term is determined by approximately solving the following  minimization problem: 
\begin{equation}
\label{eq:sol-operator-G-steklov-minimize}
\min_{\substack{w_{i}^{(1)} \in{RT}^{p}_h,~  (w_{i}^{(1)}\cdot \vec{n})|_{\partial\Omega} = \gamma v_{i}}}
\|w_{i}^{(1)}\|^2_{\Omega} + \|\Div  w_{i}^{(1)}\|^2_{\Omega} ~.
\end{equation}
The minimization problem \cref{eq:sol-operator-G-steklov-minimize} is
equivalent to the following variational problem: Find $w_{i}^{(1)} \in{RT}^{p}_h$ such that $w_{i}^{(1)}  \cdot \vec{n} = \gamma v_{i} \text{ on } {\partial\Omega}$, and
\begin{equation}
(w_{i}^{(1)} ,w_h)_{\Omega}+(\Div \,w_{i}^{(1)}, \Div \, w_h)_{\Omega} = 0\quad \forall w_h\in \text{RT}^p_{0,h}.
\end{equation}
Here, $\text{RT}^p_{0,h}:=\{ w_h \in \text{RT}^p_h ~|~  w_h\cdot \vec{n}=0 \text{ on } {\partial \Omega} \}$.

\medskip

\paragraph{Selection of $\rho$}~

Take $\rho$ as the lower eigenvalue bound of $\lambda_{n+1}$ obtained by 
 the Crouzeix--Raviart finite element method. That is,
 $$
\rho:= \frac{\lambda^{\CR}_{n+1,h}}{1+C_h^2 \lambda^{\CR}_{n+1,h}}.
 $$
Define matrices $A_0, A_1$ by
\begin{equation}
	\label{eq:LG-method-matrix-A0-A1-steklov}
A_0 := \big( M(v_{i}, v_{j}) \big)_{i,j=1}^n~,
\quad 
A_1 := \big( N( v_{i},  v_{j}) \big)_{i,j=1}^n .
\end{equation}
Note that $A_1$ is positive definite according to the choice 
 of $v_1, \cdots, v_n$.
Let $\Lambda_n$ be the maximal eigenvalue of the eigenvalue problem: $A_0 x = \Lambda A_1 x$.
With a proper selection of the 	eigenvalue index $n$ and a refined mesh, we can ensure $\Lambda_n < \rho$.

\paragraph{Matrices for the Lehmann--Goerisch theorem }
Let ${\nu_k}$ $(k=1,\cdots,n)$ be the negative eigenvalues of $Ax=\nu Bx$, where
\begin{equation}
	\label{eq:LG-method-matrix-steklov}
A_2 = \big( b_G(w_i, w_j) \big)_{i,j=1}^n, \quad
A = A_0 - \rho A_1, \quad B = A_0 - 2 \rho A_1 + \rho^2 A_2.
\end{equation}
Then we obtain the following lower bounds.
\begin{equation}\label{Lehmann--Goerisch-bounds-h}
\lambda_{k} \geq \rho - \frac{\rho}{1 - {\nu}_{n+1-k}}  \quad (1\leq k\leq n).
\end{equation}

\begin{remark}
\label{remark:steklov-LM-method-new-setting}
 Instead of taking $D=H^1(\Omega)$ as discussed in this section, we can also obtain the same eigenvalue bounds by taking 
$ {D} = \mbox{Ker}(\mathcal{K})^\perp(\subset H^1(\Omega))$ while
keeping the setting of $M,N,T,X$ the same as the one in (a) (b), and (c) in this section. 
Since the approximate eigenvector $v_{i} \in \text{CG}^p_h$ may not belong to $D=\mbox{Ker}(\mathcal{K})^\perp$, more efforts are needed to handle this issue.
Introduce the orthogonal projection $P:H^1(\Omega) \mapsto {\mbox{Ker}(\mathcal{K})^\perp}$ with respect to $M(\cdot,\cdot)$ and decompose $v_i$ by 
$$
v_i = P v_i + (I-P) v_i;\quad 
P v_i\in \mbox{Ker}(\mathcal{K})^\perp,~
(I-P) v_i \in \mbox{Ker}(\mathcal{K})~.
$$
Let $\tilde{v}_i := P v_i$.
Actually, we apply the Lehmann--Goerisch theorem with $v_i$ replaced by $\tilde{v}_i$ ($i=1,\cdots,n$). 
Note that  $\tilde{v}_i $ is just for theoretical analysis and does  not need to be computed explicitly.
Since $v_i-\tilde{v}_i \in \mbox{Ker}(\mathcal{K})$, we have $\gamma (v_i-\tilde{v}_i)=0$ on $\partial \Omega$. Hence, $\gamma v_i = \gamma \tilde{v}_i $. Thus, the $w_i$ term resulting from $\tilde{v}_i$ is the same as the one for $v_i$. In the construction of the matrices $A_0$, $A_1$, $A_2$, rather than 
$A_0$ using $v_i$ in \eqref{eq:LG-method-matrix-A0-A1-steklov}, we have $
\widetilde{A}_0 := \big(M( \tilde{u}_{i}, \tilde{u}_{j}) \big)_{i,j=1}^n$. From the orthogonality of the projection $P$ with respect to $M(\cdot,\cdot)$, we have $M(u,u)  \ge M(Pu, Pu) $ for all $u \in H^1(\Omega)$. Therefore, $A_0 - \widetilde{A}_0$ is positive semidefinite. 

Lehmann--Goerisch's theorem utilizes the following eigenvalue problem: 
$$
 \widetilde{A} x = \tilde{\nu} \widetilde{B} x \quad  (\widetilde{A} := \widetilde{A}_0 - \rho A_1, ~ \widetilde{B} := \widetilde{A}_0 - 2 \rho A_1 + \rho^2 A_2)~.
$$
Let $\widetilde{\Lambda}_n$ be the maximal eigenvalue of 
$\widetilde{A}_0 x = \widetilde{\Lambda} A_1 x$. Then
$\widetilde{\Lambda}_n \le \Lambda_n$.
By taking $(\widetilde{\Lambda}_n \le ) \: \: \Lambda_n <\rho \le \lambda_{n+1}$, 
we have that $-\widetilde{A}$ is positive definite. 
From the positive definiteness  of $\widetilde{B}$ given by Lemma \ref{lem:cond-for-positive-B}, we have that $\widetilde{B}-\widetilde{A} = \rho(\rho A_2-A_1)$ is positive definite as well.
With a proper transformation, the eigenvalue problems to determine $\tilde{\nu}$ and $\nu$ (all of which are negative) are given by
$$
(\widetilde{A}_0 - \rho A_1)x = \frac{\tilde{\nu}}{1-\tilde{\nu}} (\rho A_2 - A_1)x, \quad ({A}_0 - \rho A_1)x = \frac{{\nu}}{1-{\nu}} (\rho A_2 - A_1)x\:. 
$$
From the min-max principle for eigenvalues, we have 
$\tilde{\nu}_k \le \nu_k$ ($k=1,\cdots, n$).
Thus, we obtain the same lower bound as \cref{Lehmann--Goerisch-bounds-h}  via the  theoretical term $\tilde{\nu}_i$. That is, 
\begin{equation}\label{eq:steklov-Lehmann--Goerisch-bounds-h-from-new-setting}
\lambda_{k} \geq \rho - \frac{\rho}{1 - \tilde{\nu}_{n+1-k}} \ge \rho - \frac{\rho}{1 - {\nu}_{n+1-k}}  \quad (1\leq k\leq n)\:.
\end{equation}
	
\end{remark}

\section{Computation Results}

In this section, the Laplacian and the Steklov eigenvalue problems are considered. 
For each problem, high-precision eigenvalue bounds are obtained in two steps by applying the projection-based  eigenvalue bounds in \cref{thm:eigenvalue-explicit-bound} and Lehmann--Goerisch's  \cref{thm:Lehmann--Goerisch}. The features of the proposed algorithm are illustrated by the numerical results: the adaptivity to non-uniform meshes or graded meshes is shown in \S \ref{subsec:numerical-results-laplacian}; high-precision eigenvalue bounds and the recovery of the convergence rate for the bounds to Steklov eigenvalues are demonstrated in \S \ref{subsec:numerical-results-steklov}.

To obtain rigorous bounds for eigenvalues, we use interval arithmetic and the method based on Sylvester's law of inertia \cite{Behnke-1991} to exactly solve eigenvalue problems for matrices. 

\subsection{Dirichlet Laplacian eigenvalues}
\label{subsec:numerical-results-laplacian}

Let us consider the Dirichlet Laplacian eigenvalue problem on a dumbbell domain $\Omega$, which consists of two unit squares connected by a bar with a width of $0.40625$ (see \cref{fig:dumbbell_mesh}). The approximate values of the eigenvalues are as follows:
$$
\lambda_1 \approx 19.515307,\quad
\lambda_2 \approx 19.515448,\quad
\lambda_3 \approx 47.510661, \quad 
\lambda_4 \approx 47.519450,
$$
$$
\lambda_5 \approx 49.342953,\quad
\lambda_6 \approx 49.342955\;.
$$

The approximate eigenfunctions corresponding to these six eigenvalues are plotted in Fig. \ref{fig:dumbbell_6_eig_function} in row-wise order.

\begin{figure}[hp]
\includegraphics[width=13cm]{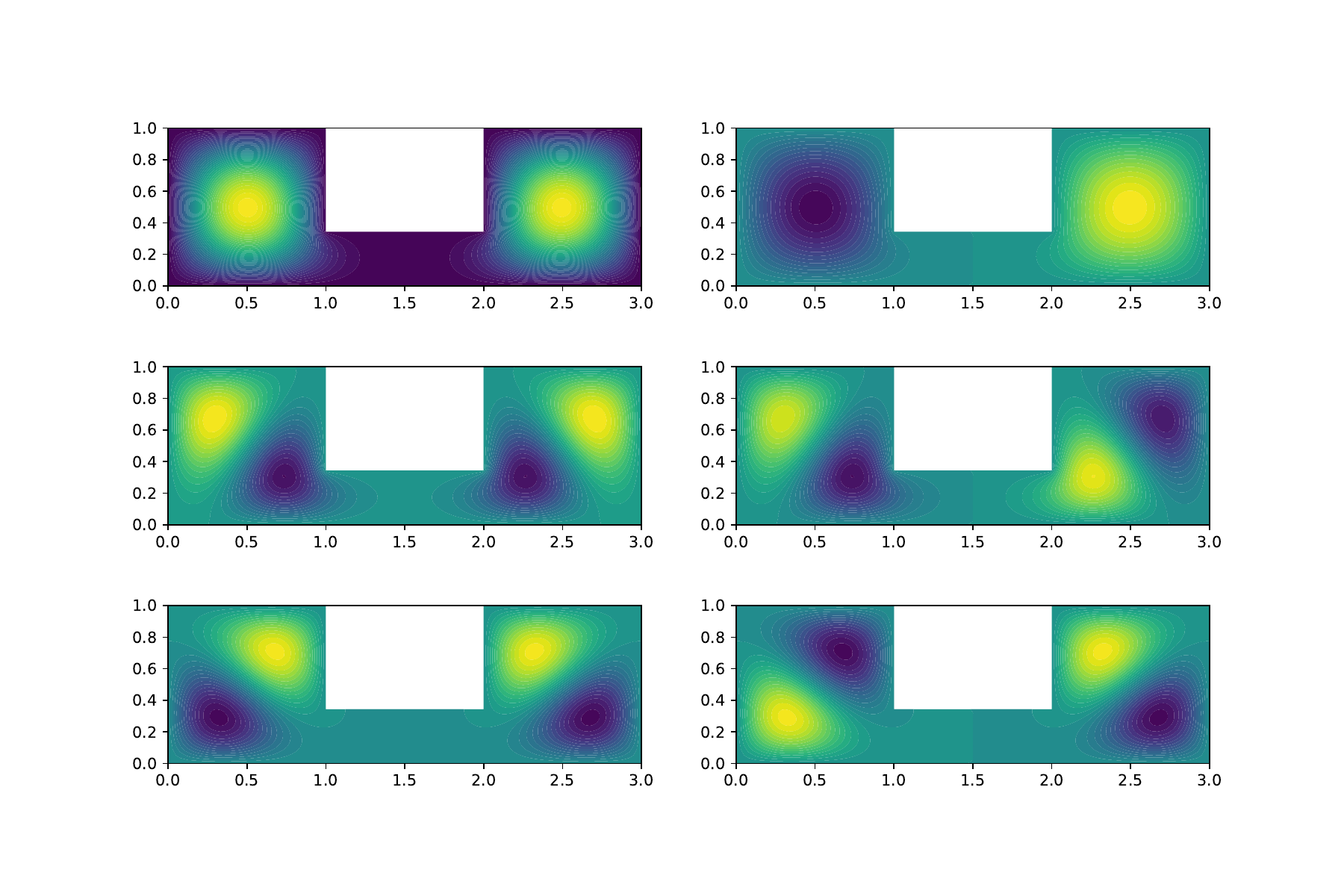} 
\caption{Approximate eigenfunctions for $\lambda_i$'s $(1\le i \le 6)$ (dumbbell domain) \label{fig:dumbbell_6_eig_function}}
\end{figure}

\begin{figure}[ht]
\begin{center}	
\includegraphics[width=10cm]{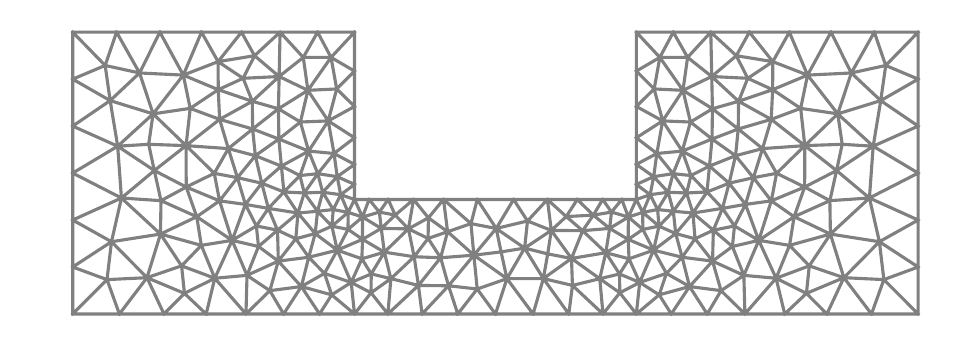} 
\end{center}
\caption{Graded mesh for the dumbbell domain \label{fig:dumbbell_mesh}}	
\end{figure}

A graded mesh is employed for the finite element computations. In accordance with the regularity of the solution, the  size of triangular elements  is chosen such that
\[
h(r) \approx 
\begin{cases}
h_{\max} \: r^{1/3}, & \text{for } r < 1, \\
h_{\max}, & \text{for } r \ge 1,
\end{cases}
\]
where \( r \) denotes the distance from an element to the two re-entrant corners of the domain, and \( h_{\max} \) is the maximum edge length of the mesh.

For a  mesh with \(h_{\max}=0.0298\), we employ the Crouzeix–Raviart FEM to compute a rigorous lower bound for \(\lambda_7\) via the projection error constant $C_h=0.1893h_{\max}$ and \cref{thm:eigenvalue-explicit-bound}:
\[
C_h \le 0.00565, \quad
\lambda_{7,h}^{\CR} \in [74.37604,74.37605 ], \quad
\lambda_7 \ge 74.19986.
\]
Thus,  for  the Lehmann--Goerisch method, the lower bound $\rho$ of \( \lambda_7 \) is set as 
$$
\rho = 74.0 <  \lambda_7.
$$
Here we deliberately take \(\rho\) well below the  lower bound for \(\lambda_7\) based on $\lambda_{7,h}^{\CR}$ to show that even a coarse estimate of \(\lambda_7\) still yields very sharp bounds for $\lambda_1, \cdots, \lambda_6$.

The rigorous lower  and upper bounds with different algorithms are displayed in Table \ref{table:bounds_laplacian_dumbbell}.
In this computation, we employ the Crouzeix--Raviart finite element method (abbreviated as \texttt{CR}), the Lagrange finite element method (\texttt{CG}), and the Lehmann--Goerisch method (\texttt{LG}) to compute these lower and upper eigenvalue bounds  shown in the table.
A sample mesh with maximum edge length (i.e., mesh size) $h_{\max} = 0.194$ is shown in Fig.~\ref{fig:dumbbell_mesh}.
By using the numerical evaluation of eigenvalues     obtained on a highly refined mesh as reference eigenvalues (denoted by $\widetilde{\lambda}_i$), we compute the errors of the lower and upper bounds, as shown in Table~\ref{table:dumbbell_eig_bounds}.

It is important to point out that the projection-based lower bounds $\lambda_i^{\CR}$ obtained in the first stage are far too inaccurate to separate $\lambda_3$ and $\lambda_4$.
 While $\lambda_1$ and $\lambda_2$ are known to be distinct, it is not a priori known whether $\lambda_3$ and $\lambda_4$ are different. 
 Here, we utilize a quite refined mesh with 1654 elements, for which the  conforming space $\text{CG}_h^3 $ has DOF=30044,  to rigorously confirm their separation by computing validated eigenvalue bounds. By applying  Lehmann--Goerisch's theorem to the cluster of eigenvalues  $\lambda_3, \cdots, \lambda_6$ and the lower bound $\rho=74\le \lambda_7$, we obtain
 \begin{equation}
\label{eq:separate_3_and_4}
\lambda_3 \le 45.6099 < 45.6495 \le \lambda_4\, . 
\end{equation}

Note that, even with a configuration that successfully separates $\lambda_3$ and $\lambda_4$, the results are still too coarse to confirm the separation of $\lambda_5$ and $\lambda_6$. Additional effort is needed to distinguish these eigenvalues (if they are different, which is not known). The shifting technique in Remark \ref{remark:Lehmann--Goerisch-with-shift} may offer a more efficient approach, as it avoids the need to rigorously solve the large linear system in \eqref{eq:discrete-equ-laplacian}.

\begin{table}
\renewcommand{\arraystretch}{1.2}
\begin{center}
\caption{\label{table:bounds_laplacian_dumbbell} Lower and upper bounds for Laplacian eigenvalues (dumbbell domain, $h_{\max} =0.290$, number of elements=177). }
	\begin{tabular}{|c|c|c|c|c|c|}
\hline
 \multicolumn{1}{|c|}{\quad} & \multicolumn{3}{|c|}{\rule[-0.2cm]{0.cm}{0.6cm}
Lower bound: $\underline{\lambda}_i$} & \multicolumn{2}{|c|}{\rule[-0.2cm]{0.cm}{0.6cm}
Upper bound: $\overline{\lambda}_i$}\\
\cline{2-6}
$\lambda_i$	 & \texttt{CR} & \texttt{LG}(p=1) &\texttt{LG}(p=2) &~ \texttt{CG}(p=1) ~ &\texttt{CG}(p=2) \\
\hline
$\lambda_1$ &  $18.5595$ & $19.2865$ & $19.5080$ & $20.1365$ & $19.5265$ \\ \hline 
$\lambda_2$ &  $18.5596$ & $19.2867$ & $19.5082$ & $20.1366$ & $19.5266$ \\ \hline 
$\lambda_3$ &  $41.9488$ & $41.1170$ & $47.2888$ & $51.0975$ & $47.6105$ \\ \hline 
$\lambda_4$ &  $41.9554$ & $41.1189$ & $47.2978$ & $51.1087$ & $47.6192$ \\ \hline 
$\lambda_5$ &  $43.7222$ & $41.5923$ & $49.2050$ & $53.3572$ & $49.4143$ \\ \hline 
$\lambda_6$ &  $43.7222$ & $41.5926$ & $49.2050$ & $53.3572$ & $49.4143$ \\ \hline \end{tabular}
\end{center}
\end{table}

\begin{table}[pt]
\caption{Errors of eigenvalue bounds shown in Table \ref{table:bounds_laplacian_dumbbell} (dumbbell domain, w.r.t. $\widetilde{\lambda}_i$)\label{table:dumbbell_eig_bounds}}
\begin{center}
\renewcommand{\arraystretch}{1.1}
\begin{tabular}{|c|c|c|c|c|c|}
\hline
 \multicolumn{1}{|c|}{\quad} & \multicolumn{3}{|c|}{\rule[-0.2cm]{0.cm}{0.6cm}
Error of lower bound: $\widetilde{\lambda}_i-\underline{\lambda}_i$} & \multicolumn{2}{|c|}{\rule[-0.2cm]{0.cm}{0.6cm}
Error of upper bound: $\overline{\lambda}_i-\widetilde{\lambda}_i$}\\
\cline{2-6}
	 & \texttt{CR} & \texttt{LG}(p=1) &\texttt{LG}(p=2) &~ \texttt{CG}(p=1) ~ &\texttt{CG}(p=2) \\
\hline
$\lambda_1$, $\lambda_2$ & \texttt{4.9e-02}&\texttt{1.2e-02}&\texttt{3.7e-04}&\texttt{3.2e-02}&\texttt{5.7e-04}\\ \hline 
$\lambda_3$, $\lambda_4$ & \texttt{1.2e-01}&\texttt{1.3e-01}&\texttt{4.7e-03}&\texttt{7.5e-02}&\texttt{2.1e-03}\\ \hline 
$\lambda_5$, $\lambda_6$  & \texttt{1.1e-01}&\texttt{1.6e-01}&\texttt{2.8e-03}&\texttt{8.1e-02}&\texttt{1.4e-03}\\ \hline 
\end{tabular}
\end{center}
\end{table}

\begin{figure}[ht]
\begin{center}
\includegraphics[width=8cm]{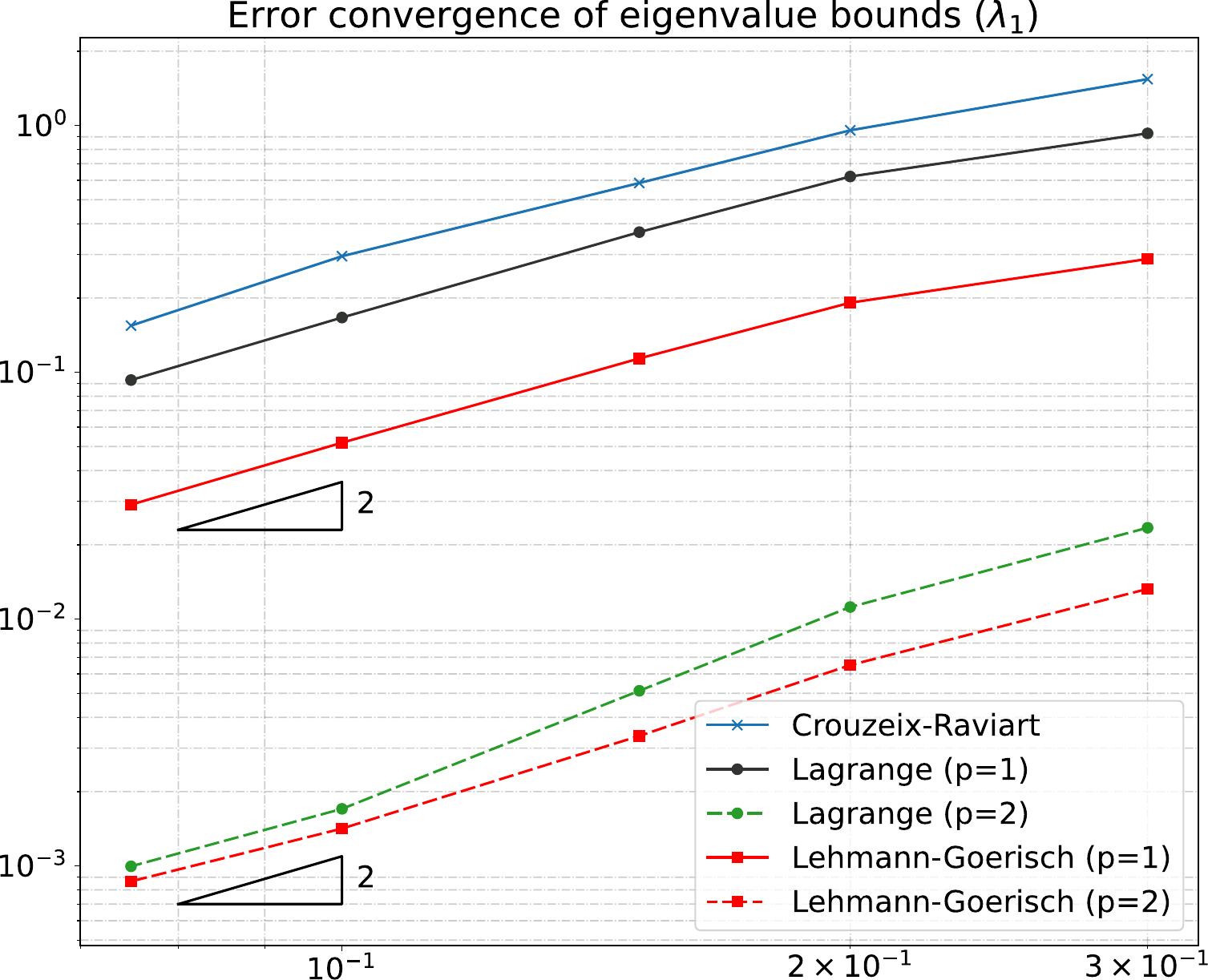}\\
\vskip 0.1cm
\includegraphics[width=8cm]{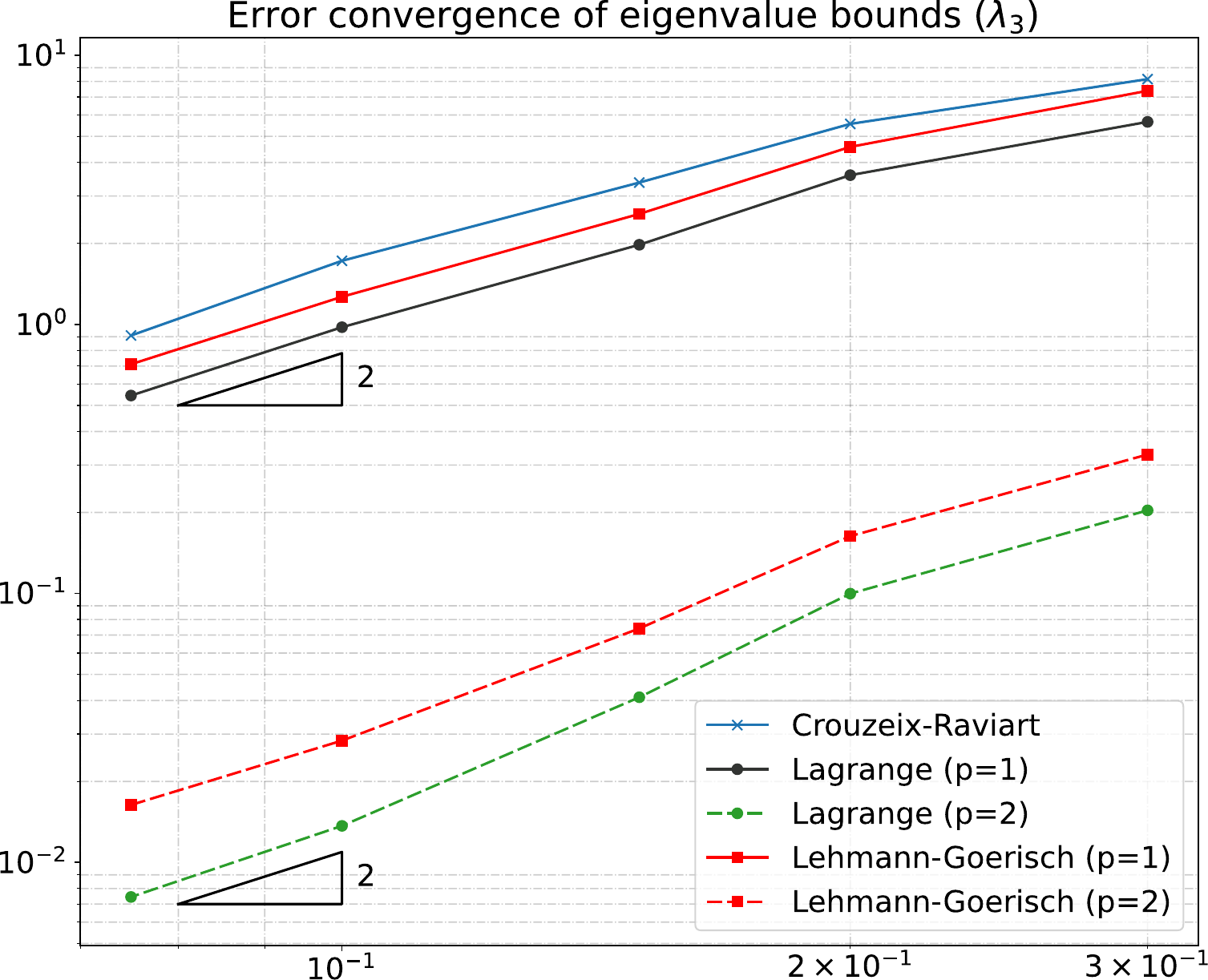} \\
\vskip 0.1cm
\includegraphics[width=8cm]{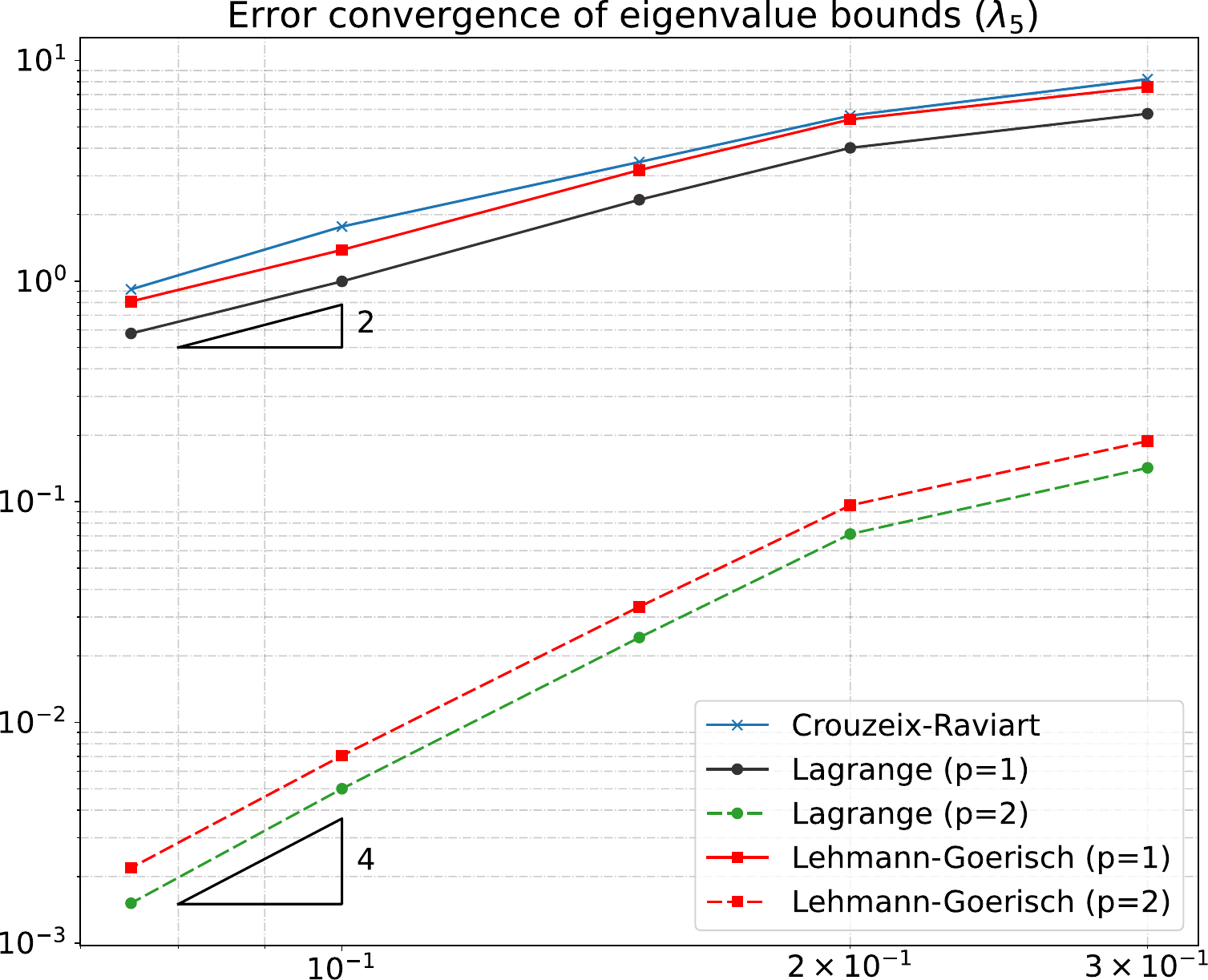}
\end{center}
\caption{Convergence of Laplacian eigenvalue bounds for dumbbell domain ($x$-axis: mesh size $h_{\max}$; $y$-axis: magnitude of the error). \label{fig:convergence-dumbbell-laplacian}}
\end{figure}

 \medskip

For recursively refined meshes,  the error convergence of various bounds for the eigenvalues is shown in Fig. \ref{fig:convergence-dumbbell-laplacian}.
From the numerical results, the following features of the proposed algorithm are observed.

\begin{enumerate}
\item For the eigenfunctions corresponding to $\lambda_1, \ldots, \lambda_4$, which may lack full $H^2$-regularity, the use of a graded mesh enables the associated eigenvalue bounds computed with linear finite elements to achieve the optimal convergence rate of $O(h^2)$.
\item When the eigenfunctions possess higher regularity (as one might guess for the eigenfunctions corresponding to $\lambda_5$ and $\lambda_6$ based on the numerical plots in Fig. \ref{fig:dumbbell_6_eig_function}), the lower bounds obtained via the Lehmann--Goerisch method and the upper bounds from the Lagrange finite element method exhibit optimal convergence rates of $O(h^2)$ for $p=1$ and $O(h^4)$ for $p=2$.

\item The Lehmann--Goerisch method using the conforming finite element method with polynomial degree $p=2$ (i.e., $\text{CG}^2_h$) yields significantly sharper lower bounds for the eigenvalues compared to those obtained using the Crouzeix--Raviart finite element method.

\end{enumerate}

\subsection{Steklov eigenvalues}
\label{subsec:numerical-results-steklov}

In this section we consider two model Steklov eigenvalue problems: one on the unit square
\[
\Omega = (0,1)\times(0,1),
\]
and one on the L-shaped domain
\[
\Omega = (0,2)\times(0,2)\setminus [1,2)\times[1,2).
\]
For each example we obtain a rough lower bound from \cref{thm:eigenvalue-explicit-bound} and a high-precision lower bound via Lehmann--Goerisch’s theorem.

\begin{figure}[p]
\begin{center}
\includegraphics[width=4cm]{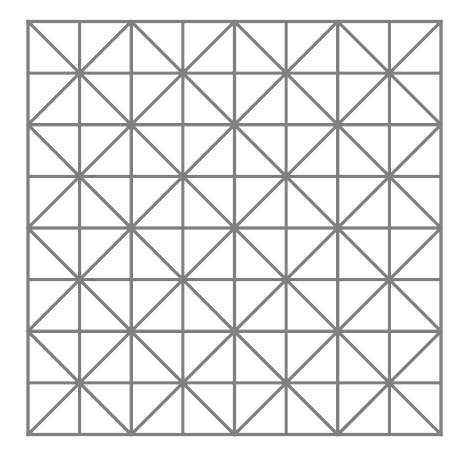} \quad \includegraphics[width=4cm,height=4cm]{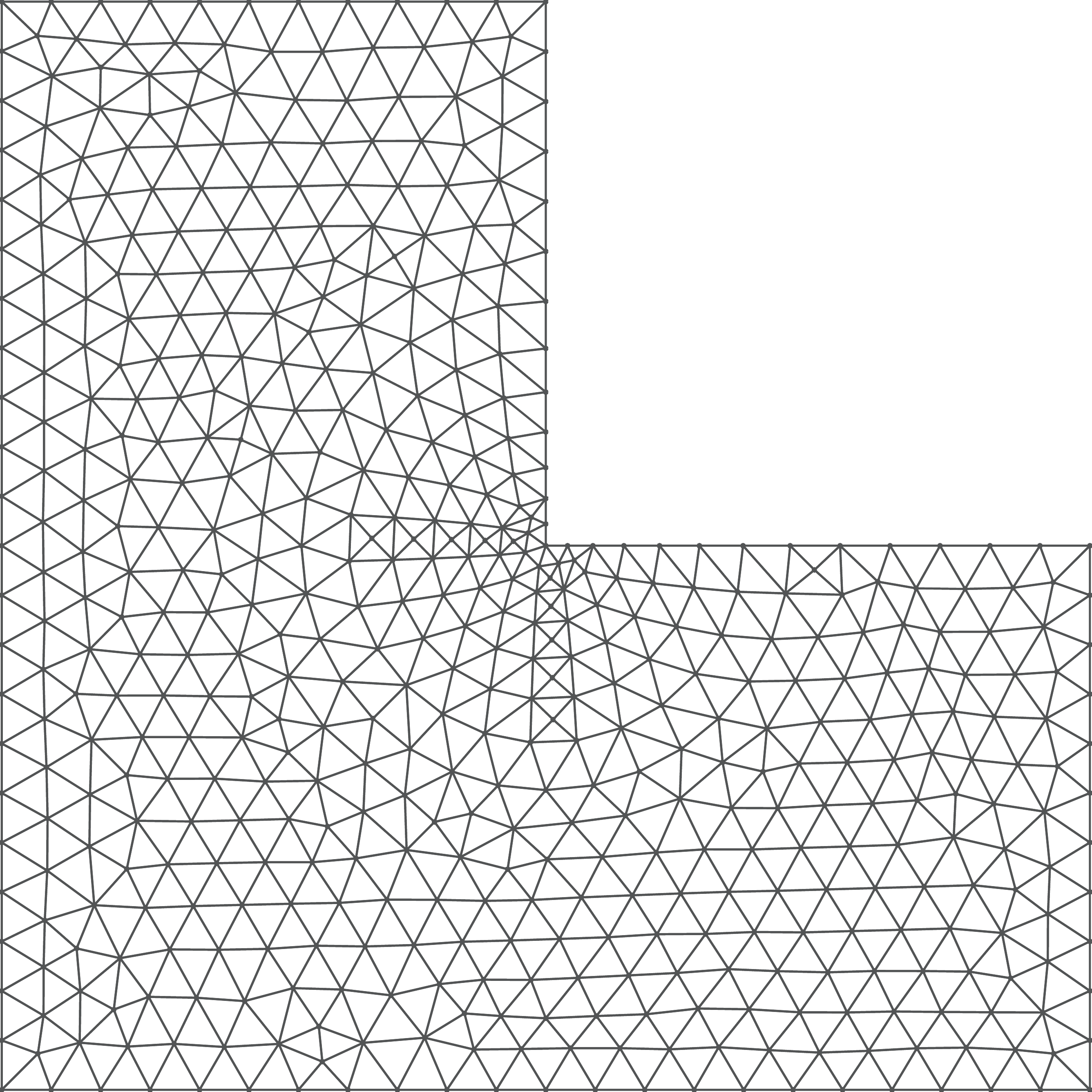}
\caption{Uniform mesh for the square domain and 
graded mesh for the L-shaped domain  \label{fig:mesh_steklov}}	
\end{center}
\end{figure}

\paragraph{\bf Square domain}
By separation of variables, this eigenvalue problem can be solved in a closed form, up to the solution of a system of two transcendental equations with two unknowns. These closed-form solutions show that $\lambda_2=\lambda_3$.

\begin{figure}[p]
\begin{center}
\includegraphics[width=3cm]{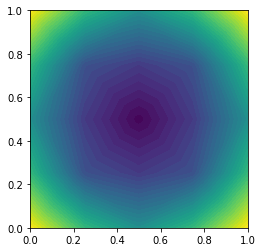}
\includegraphics[width=3cm]{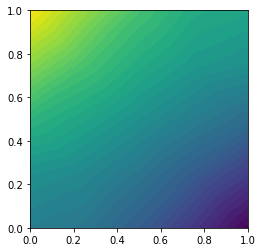}
\includegraphics[width=3cm]{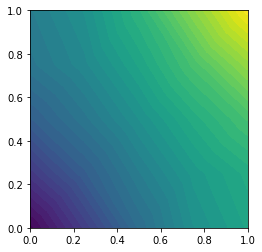}
\includegraphics[width=3cm]{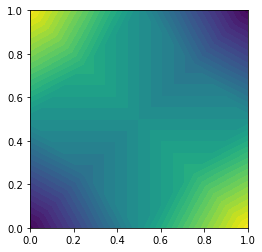}
\end{center}
\caption{Approximate eigenfunctions for Steklov eigenvalues $\lambda_i (1\le i \le 4)$ (square domain).\label{fig:eig-func-square-steklov}}
\end{figure}

We triangulate \(\Omega\) uniformly with an \(N\times N\) grid; see \cref{fig:mesh_steklov} for \(N=8\).  The mesh size \(h_{\max}\) as the maximum edge length is \(h_{\max}=\sqrt{2}/8\) when \(N=8\).
Using conforming  FEM with $\text{CG}^2_h$ and $N=32$,   
we obtain numerical values for the first four Steklov eigenvalues (see also, e.g.,  \cite{YANG20092388}) :
\[
\lambda_1\approx 0.24007909,\quad
\lambda_2 = \lambda_3\approx 1.49230326,\quad
\lambda_4\approx 2.08264706\,.
\]
The eigenfunctions for these 4 eigenvalues are plotted in \cref{fig:eig-func-square-steklov}.

For a mesh with \(N=64\), we employ the Crouzeix–Raviart FEM to compute a lower bound for \(\lambda_4\) via the projection error constant in \cref{eq:C_h_definition_steklov} and \cref{thm:eigenvalue-explicit-bound}:
\[
C_h \le 0.12717,\quad
\lambda^{\CR}_{4,h} \in [2.08216,2.08217], \quad 
\lambda_4 \ge 2.01433.
\]
Hence, in applying Lehmann–Goerisch’s theorem we take $n=3$ and 
\begin{equation}
	\label{eq:lg-setting-steklov}
\lambda_3 < 1.49246 \le \rho = 2.0 <  \lambda_4.
\end{equation}

Figure~\ref{fig:convergence-square-steklov} shows the convergence of the computed eigenvalue bounds with respect to high-precision  reference values.  Note that the results for $\lambda_2$ and $\lambda_3$ are almost the same.  We observe that, compared with the convergence rate $O(h)$ of the Crouzeix--Raviart method (see also the theoretical discussion in Remark \ref{remark:sub-optimal-rate-steklov}), the lower bounds from the \(P_1\) conforming FEM combined with Lehmann--Goerisch’s theorem recover the \(O(h^2)\) rate, matching the \(O(h^2)\) rate of the upper (Rayleigh--Ritz) bounds from the Lagrange FEM.

\begin{figure}[ht]
\begin{center}
\includegraphics[width=9cm]{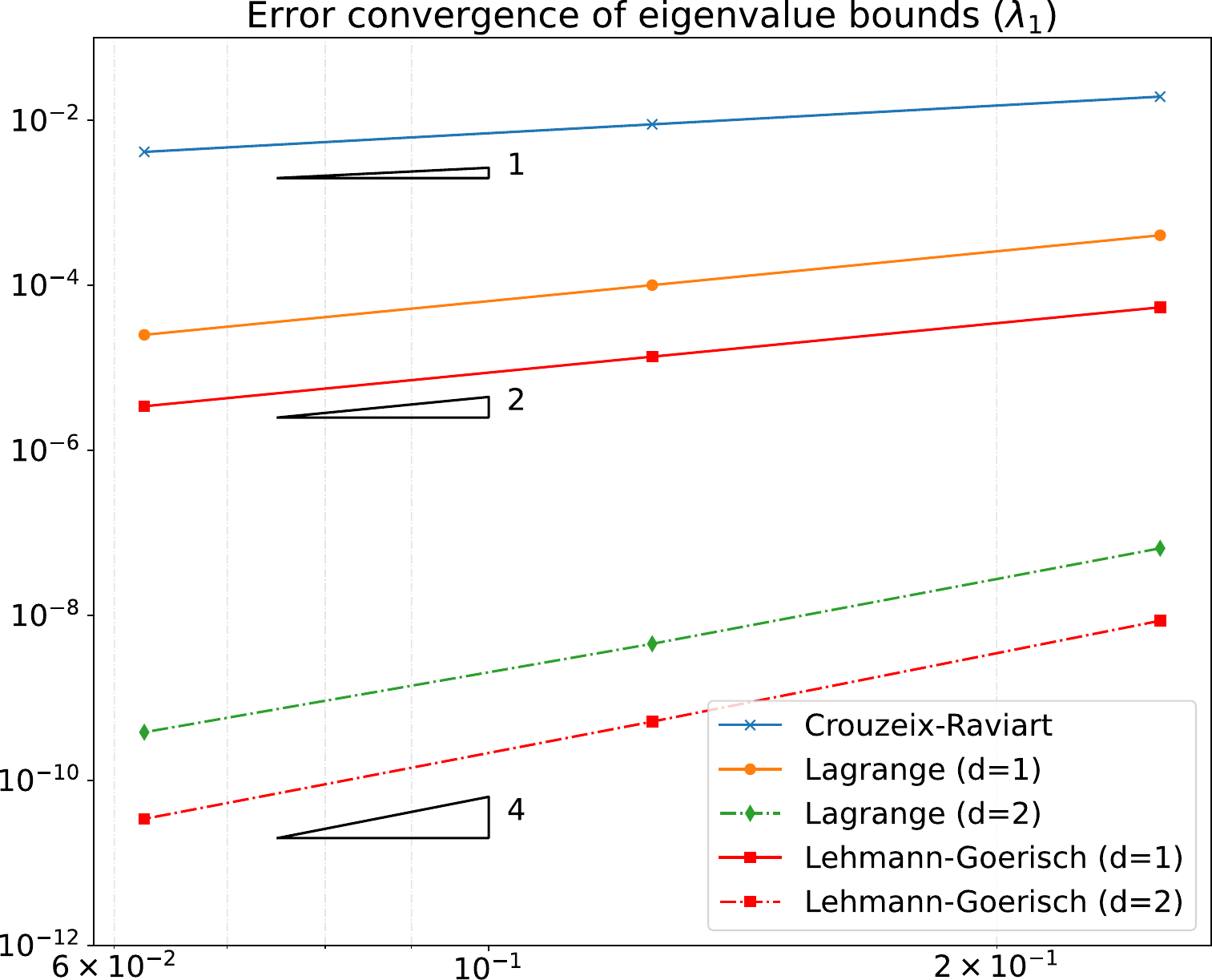} \\
\vskip 0.5cm
\includegraphics[width=9cm]{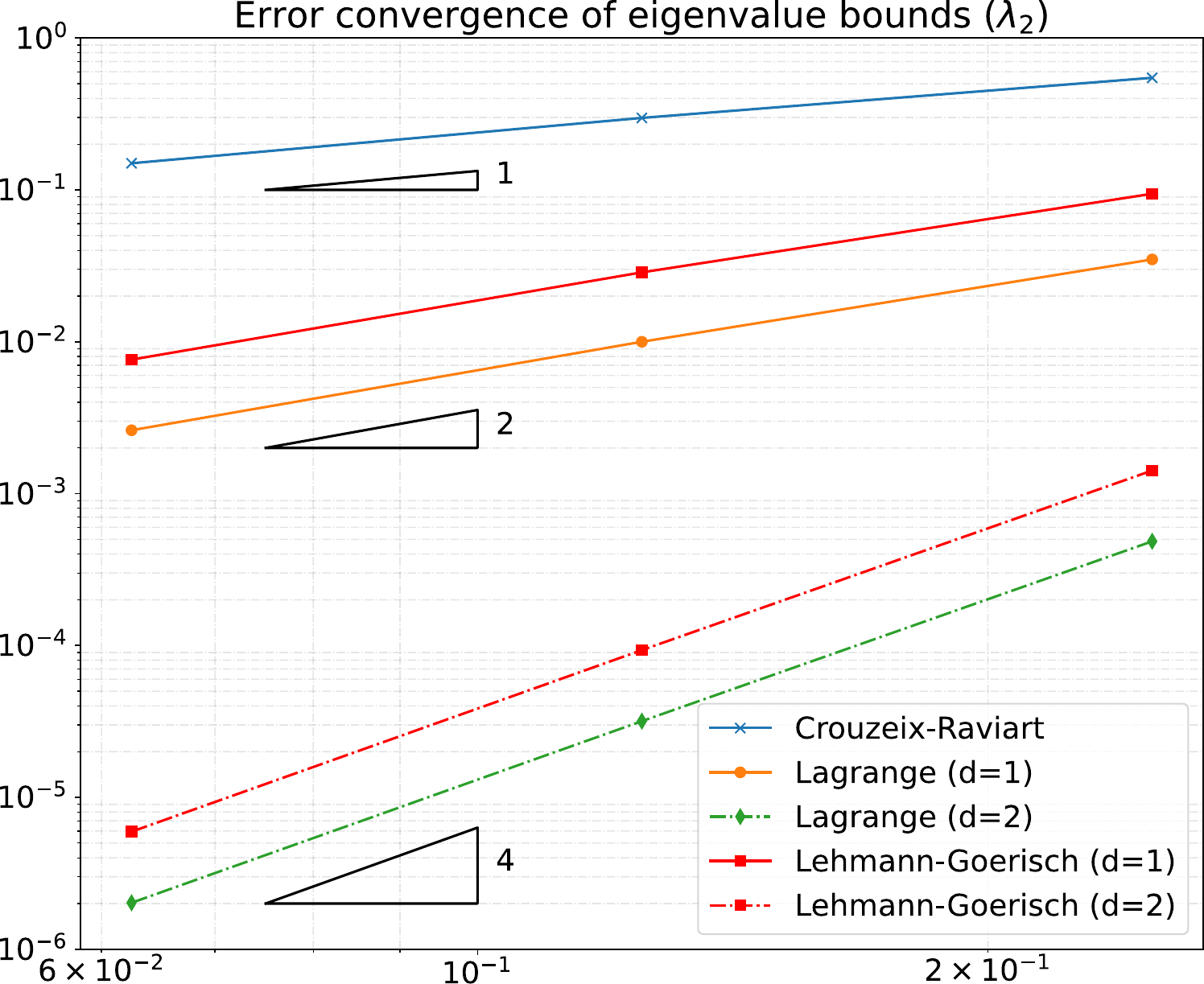}
\end{center}
\caption{Error convergence of Steklov eigenvalue bounds on unit square domain ($x$-axis: mesh size; $y$-axis: magnitude of the error). \label{fig:convergence-square-steklov}}
\end{figure}

\medskip

\paragraph{\bf L-shaped domain}

\begin{figure}[ht]
  \centering
  \includegraphics[width=8cm]{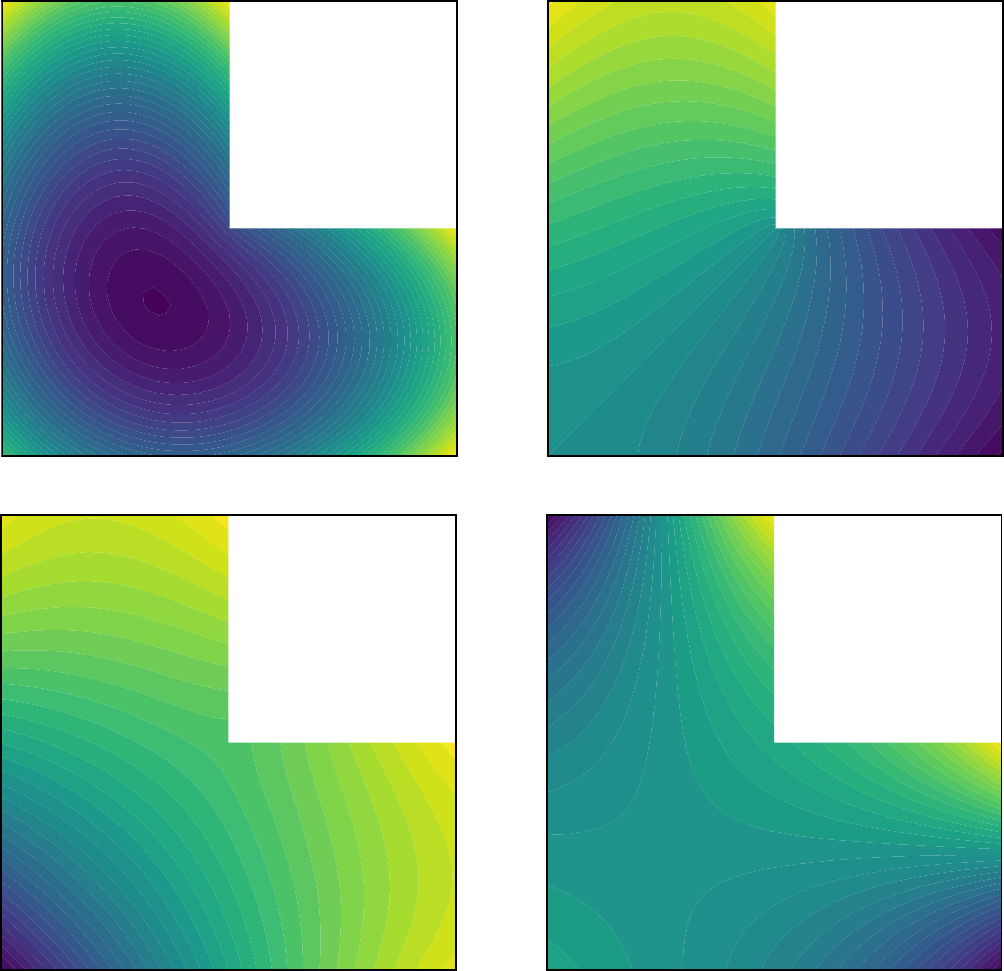}
  \caption{Approximate eigenfunctions for Steklov eigenvalue $\lambda_i$'s $(1\le i\le 4)$(L-shaped domain)}
  \label{fig:L_4_steklov_eig_function}
\end{figure}

Figure \ref{fig:L_4_steklov_eig_function} displays the computed eigenfunctions for the first four eigenvalues on the L-shaped domain.
The first, third, and fourth approximate eigenfunctions are symmetric with respect to reflection across the axis $x_1 = x_2$, whereas the second is antisymmetric.

We use a geometrically graded mesh with \(h_K=O(r^{1/3})\), where \(r\) is the distance from element \(K\) to the re-entrant corner (see \cref{fig:mesh_steklov}). 

For a graded mesh with \(h_{\max} = 0.031\), we employ the Crouzeix–Raviart FEM to obtain the following quantities:\[
C_h \le 0.14902,\quad
\lambda_{4,h} \in [1.69171,1.69172], \quad
\rho =1.6<1.63045 \le \lambda_4.
\]

Figure~\ref{fig:convergence-L-steklov} shows the convergence of the eigenvalue bounds with respect to high-precision reference values. 
Note that the lower eigenvalue bounds obtained by the Crouzeix–Raviart method exhibit only $O(h)$ convergence, which is due to the projection error term $C_h = O(\sqrt{h})$, as explained in Remark \ref{remark:sub-optimal-rate-steklov}.
The observed convergence rate of \( O(h^2) \) for the eigenvalue bounds of \( \lambda_2 \) with \( p = 2 \) indicates that the eigenfunction associated with \( \lambda_2 \) lacks the higher regularity possessed by those corresponding to \( \lambda_1 \) and \( \lambda_3 \). On the graded mesh, the $P_1$ conforming FEM achieves the expected $O(h^2)$ rate for both lower and upper bounds. In contrast, the $P_2$ conforming FEM exhibits a reduced convergence rate of $O(h^2)$ for the second eigenvalue, while retaining $O(h^4)$ convergence for the first and third eigenvalues.
Based on this observation, one might guess that 
the eigenfunctions corresponding to the first and third eigenvalues have $H^3$ regularity. 

\medskip

\begin{figure}[ht]
\begin{center}
\includegraphics[width=8.0cm]{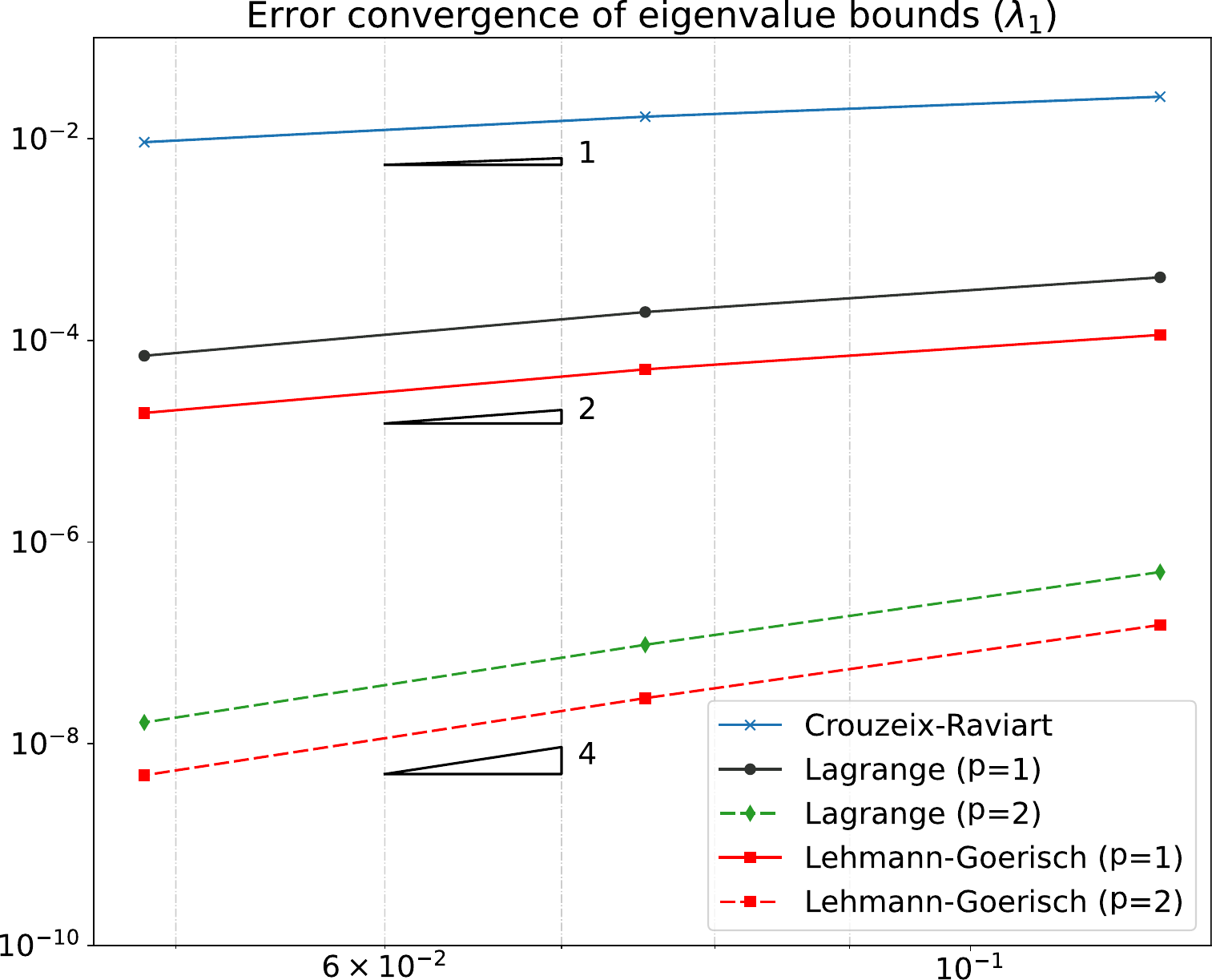} \\
\vskip 0.1cm
\includegraphics[width=8.0cm]{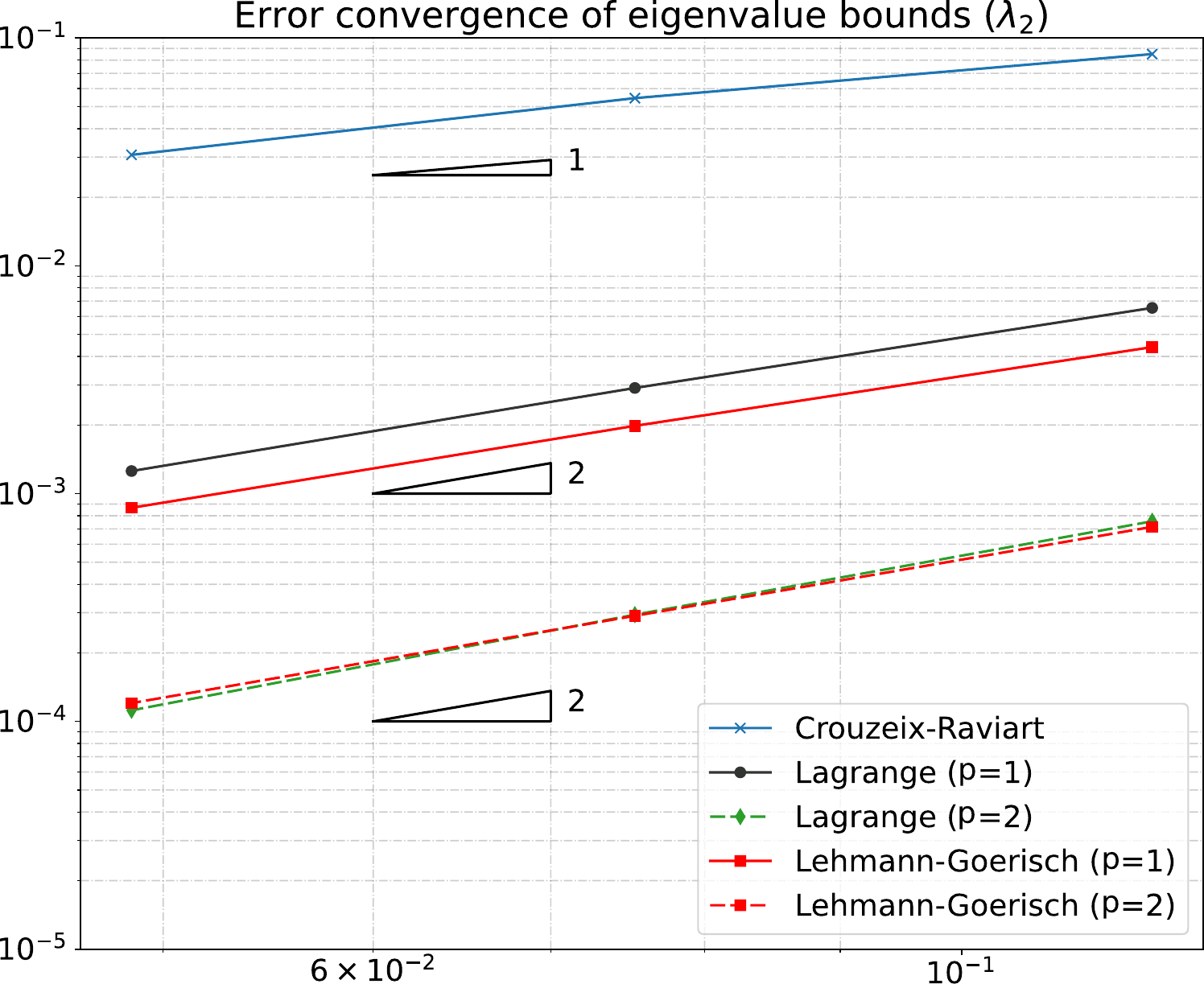}\\
\vskip 0.1cm
\includegraphics[width=8.0cm]{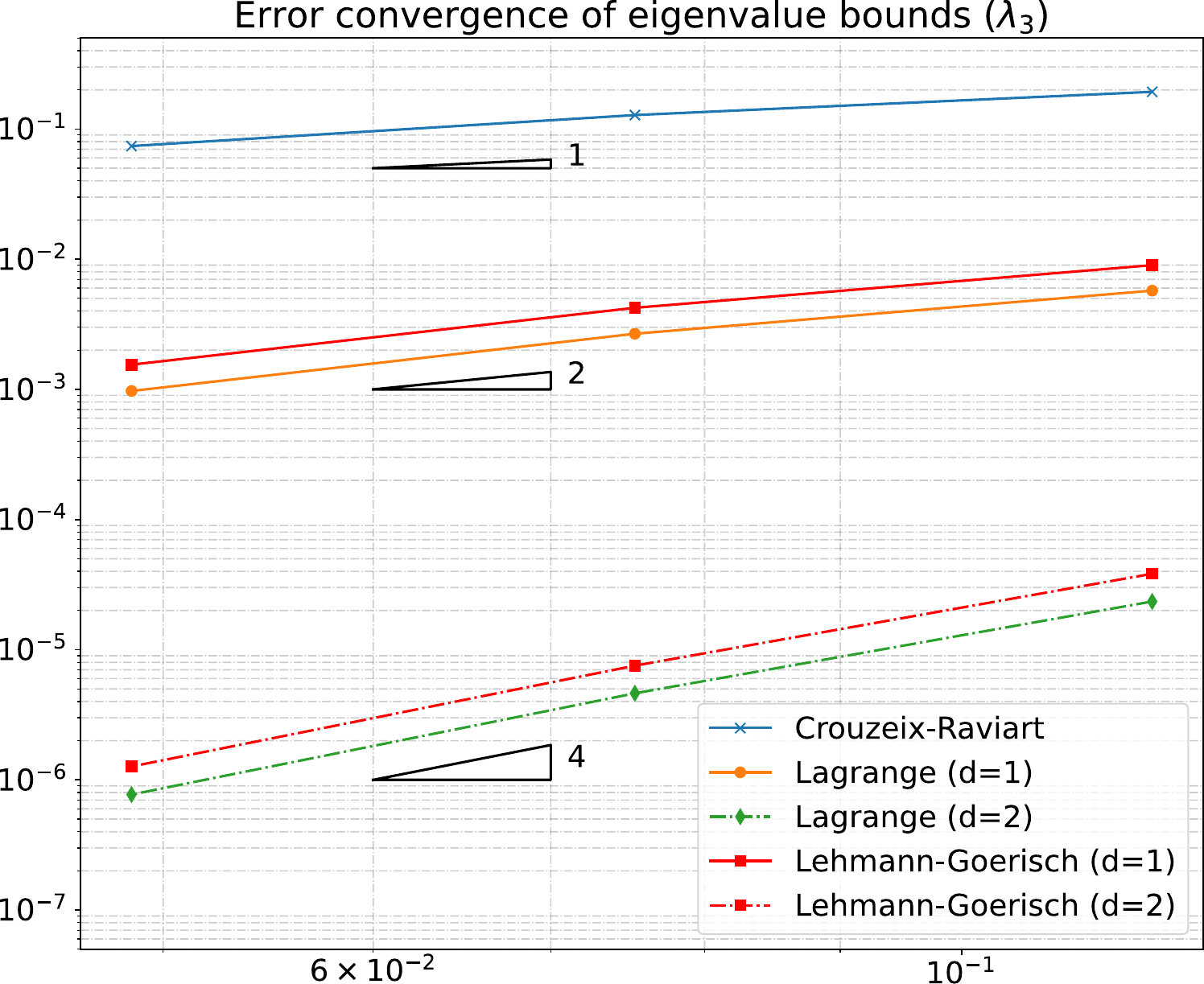}
\end{center}
\caption{Error convergence of Steklov eigenvalue bounds on the L-shaped domain ($x$-axis: mesh size $h_{\max}$; $y$-axis: magnitude of the error). \label{fig:convergence-L-steklov}}
\end{figure}

\medskip

\paragraph{Discussion on rigorous eigenvalue bounds}

To control floating-point and rounding errors, we employ interval arithmetic via INTLAB \cite{Rump-1999} together with Behnke’s method \cite{Behnke-1991}, as mentioned at the beginning of this section.  This provides fully verified eigenvalue bounds for the resulting generalized matrix eigenproblems.  The verified Steklov eigenvalue bounds based on the setting \eqref{eq:lg-setting-steklov} are listed in Tables \ref{table:square-stekolv-rigorous-bounds} and \ref{table:L-stekolv-rigorous-bounds}.

Note that the selection of an eigenvalue cluster when applying the Lehmann--Goerisch method affects the accuracy of the lower bounds.  
For example, in the case of Steklov eigenvalues, by choosing  
\[
\lambda_1 < \rho = 1.45 \le \lambda_2,
\]
on an \(8 \times 8\) triangulation of the square domain, and using the second-order Lagrange finite element space \( \text{CG}_h^2 \), we obtain the following rigorous bounds:
\[
\underline{0.2400790}83 \le \lambda_1 \le \underline{0.2400790}91; \quad 
\overline{\lambda}_1 - \underline{\lambda}_1 = 8.00\texttt{E}{-09}.
\]
This result is significantly sharper than the one presented in Table~\ref{table:square-stekolv-rigorous-bounds}, which was computed using a finer \( 32 \times 32 \) triangulation.

\begin{table}[h]
\caption{Rigorous bounds for the Steklov eigenvalues with $p=2$ and $N=32$ (square domain) \label{table:square-stekolv-rigorous-bounds}}
\begin{center}
	\begin{tabular}{|c|c|}
 \hline
\rule[-0.2cm]{0cm}{0.60cm}{}
 Rigorous bound & Bound gap: $\overline{\lambda}_i - \underline{\lambda}_i$\\
 \hline
\rule[-0.2cm]{0cm}{0.60cm}{}
$\underline{0.2}3993698\le \lambda_1 \le \underline{0.2}4007908546 $ 
& \texttt{1.42E-04}
\\
\hline
\rule[-0.2cm]{0cm}{0.60cm}{} 
$\underline{1.492}04771\le \lambda_2 \le \underline{1.492}303245657 $
& \texttt{2.56E-04}
\\
\hline
\rule[-0.2cm]{0cm}{0.60cm}{} 
$\underline{1.492}04789\le \lambda_3 \le \underline{1.492}303278880$ &
\texttt{2.56E-04}
\\
\hline
\end{tabular}
\end{center}
\end{table}

%


\begin{table}[h]
\caption{Rigorous bounds for the Steklov eigenvalues with $p=2$ and $h=0.118$ (L-shaped domain) \label{table:L-stekolv-rigorous-bounds}}
\begin{center}
\begin{tabular}{|c|c|}
 \hline
\rule[-0.2cm]{0cm}{0.60cm}{}
Rigorous bound & Bound gap:
 $\overline{\lambda}_i - \underline{\lambda}_i$\\
 \hline
\rule[-0.2cm]{0cm}{0.60cm}{}
$\underline{0.34141}4120 \le \lambda_1 \le  \underline{0.34141}605885$ 
& \texttt{1.93E-06}  \\
 \hline
\rule[-0.2cm]{0cm}{0.60cm}{} 
$\underline{0.616}711190 \le \lambda_2 \le \underline{0.616}98333496$ 
& \texttt{2.72E-04}  \\
 \hline
\rule[-0.2cm]{0cm}{0.60cm}{} 
$\underline{0.98427}4379  \le \lambda_3 \le \underline{0.98427}919894$ 
& \texttt{4.83E-06}  \\
  \hline
\end{tabular}
\end{center}
\end{table}

\section*{Summary}
In this paper, a two-stage algorithm is proposed to provide high-precision guaranteed eigenvalue bounds for differential operators. The application of this algorithm to the Dirichlet Laplacian eigenvalue problem and the Steklov eigenvalue problem demonstrates its adaptivity to graded meshes, high-order FEMs and provides high-precision eigenvalue bounds. It is worth pointing out that in addition to second-order differential operators, the proposed algorithm can also be applied to higher order differential operators, for example, the biharmonic operator.  Early approaches for biharmonic operators can be found in \cite{goerisch1985eigenwertschranken,wieners1997bounds} and more detailed discussions will be provided in our subsequent work.

\bibliography{library}

\end{document}